\documentclass{article} 

\usepackage[utf8]{inputenc}
\usepackage[margin=1.06in]{geometry}
\usepackage{amsmath,amsthm}
\usepackage{mathtools}
\usepackage{amssymb}
\usepackage{amsopn}
\usepackage{bbm}
\usepackage{color}
\usepackage{hyperref}
\usepackage{url}
\usepackage{tikz}
\usepackage{tabu}
\usepackage[font=footnotesize,labelfont=bf]{caption}
\usepackage{enumerate}
\usepackage{placeins}
\usepackage{eso-pic}
\usepackage{caption}

\usepackage{mathrsfs, stmaryrd}
\usepackage{a4wide}

\usepackage{hyperref}

\newtheorem{theorem}{Theorem}[section]
\newtheorem{corollary}[theorem]{Corollary}
\newtheorem{definition}[theorem]{Definition}
\newtheorem{lemma}[theorem]{Lemma}
\newtheorem{remark}[theorem]{Remark}
\newtheorem{proposition}[theorem]{Proposition}
\newtheorem{example}[theorem]{Example}

\newtheorem{assumption}[theorem]{Assumption}

\def\ba#1\ea{\begin{align*}#1\end{align*}}

\def \follmer {F\"{o}llmer}

\newcommand{\R}{\mathbb{R}}
\newcommand{\E}{\mathbb{E}}
\newcommand{\N}{\mathbb{N}}
\newcommand{\1}{\mathbf{1}}

\newcommand{\bb}[1]{\mathbb{#1}}

\newcommand{\assign}{:=}
\newcommand{\dd}{\mathrm{d}}
\newcommand{\op}[1]{\ensuremath{\operatorname{#1}}}

\DeclareMathOperator*{\argmax}{argmax}
\hypersetup{
  colorlinks = true, 
  urlcolor = blue, 
  linkcolor = blue, 
  citecolor = red 
}
\title{Fractional  It\^o Calculus}
\author{Rama Cont and Ruhong Jin\\
Mathematical Institute, University of Oxford}
\date{2021}

\begin{document}

\maketitle

\begin{abstract}
We derive  It\^o-type change of variable formulas for smooth functionals of irregular paths with non-zero $p-$th variation along a sequence of partitions where $p \geq 1$ is arbitrary, in terms of fractional derivative operators, extending the results of the F\"ollmer-Ito calculus to the general case of paths with 'fractional' regularity.
In the  case where $p$ is not an integer, we show that    the   change of variable formula may sometimes contain a non-zero a 'fractional' It\^o remainder term and provide a representation for this remainder term. 
These results are then extended to 
paths with non-zero $\phi-$variation and multi-dimensional paths.
Using these results, we derive an isometry property for the pathwise F\"ollmer integral in terms of $\phi$ variation.   
\end{abstract}


\tableofcontents

\newpage
Hans \follmer\ derived in \cite{follmer1981} a pathwise formulation of the Ito formula and laid the grounds for the development of a pathwise approach to Ito calculus, which has been developed in different directions in \cite{ananova2017,bertoin1987,chiu2018,chiu2021,cont2012,cont2018pathwise,contdas2017,davis2018, obloj2021,schied2016}.

\follmer's original approach focuses on functions of paths with finite quadratic variation along a sequence of partitions.
In a recent work \cite{cont2018pathwise},
Cont and Perkowski  extended the \follmer-Ito formula of \cite{follmer1981} to function(al)s of paths with variation of order $p\in 2\mathbb{N}$ along a sequence of partitions and obtained functional change of variable formulas, applicable  to functionals of  Fractional Brownian motion and other fractional processes with arbitrarily low regularity (i.e. any Hurst exponent $H>0$). These results involve pathwise integrals defined as limits of compensated left Riemann sums, which are in turn related to rough integrals associated with a reduced order-p rough path \cite{cont2018pathwise}.

As the notion of $p$-the order variation may be defined for any $p>0$, an interesting question is to investigate how the results in \cite{cont2018pathwise} extend to 'fractional' case   $p\notin \mathbb{N}$. In particular  one may ask whether the  change of variable formula contains  a fractional remainder term in this case and whether the definition of the compensated integral needs to be adjusted.

We investigate these questions using the tools of fractional calculus \cite{samko1987fractional}. Given that fractional derivative operators are  (non-local) integral operators, one challenge is to obtain non-anticipative, 'local' formulas which have similar properties to those obtained in the integer case \cite{cont2018pathwise}. We are able to do so using a 'local' notion of fractional derivative and 
 exhibit conditions under which these change of variable formulas contain (or not) a 'fractional It\^o remainder term'. It is shown that in most cases there is no remainder term; we also discuss some cases where a non-zero remainder term appears and give a representation for this  term.

These results are first derived for smooth functions then extended to functionals, using the concept of vertical derivative \cite{CF10B}. We extend these results to the case of paths with finite $\phi$-variation \cite{han2021probabilistic}  for a class of  functions $\phi$ and we obtain an isometry formula for the pathwise integral in terms of $\phi$-variation, extending the results of \cite{ananova2017,cont2018pathwise} to the fractional case. 
Finally, we extend these results to the multi-dimensional case. 

Our change of variable formulas are purely analytical and pathwise in nature,  but applicable to functionals of fractional Brownian motions and other fractional processes with arbitrary Hurst exponent, leading in this case to non-anticipative 'It\^o' formulas for functionals of such processes. 
 However,  as probabilistic assumptions play no role in the derivation of our results,  we have limited to a minimum the discussion of such examples.

\paragraph{Outline}
Section \ref{sec.prelim}
 recalls some results on pathwise calculus for functions of irregular paths (Sec. \ref{sec.pvariation}) and  fractional derivative operators and  associated fractional Taylor expansions (Sec. \ref{sec.fractionalderivative}).

Section \ref{sec.fractionalIto} contains our main results on change of variable formulas for function(al)s of paths with fractional regularity. We first give a change of variable formula without remainder term for time-independent functions (Theorem \ref{itoformula1}), followed by a discussion of an example where a remainder term may appear (Example \ref{nonzero}). We then provide a formula for computing this fractional remainder term using an auxiliary space(Theorem \ref{EquiDistantFormula}).
Sections \ref{sec.timedependent} extends these results to the time-dependent case and Section \ref{sec.pathdependent} extends them to the case of path-dependent functionals using the Dupire derivative.

In Section \ref{sec.phi} we show that these results may be extended to the case where the $p-$th  variation is replaced by the more general concept of $\phi-$variation \cite{han2021probabilistic}.
In Section \ref{sec.isometry} we derive a pathwise isometry formula extending a result of \cite{ananova2017} to the case of $\phi-$variation.
Finally, in Section \ref{sec.multidimensional} we discuss extensions to the multidimensional case. These extensions are not immediate, as the space $V_p(\pi)$ is not a vector space.

\section{Preliminaries} \label{sec.prelim}
  
\subsection{Pathwise calculus for paths with finite p-th variation}\label{sec.pvariation}

We define, following \cite{cont2018pathwise,follmer1981}, the concept of $p$-th variation along a sequence of partitions $\pi_n=\{t_0^n, \dots, t^n_{N(\pi_n)}\}$ with $t_0^n=0<...< t^n_k<...< t^n_{N(\pi_n)}=T$. Define the \emph{oscillation} of $S \in C([0,T],\R)$ along $\pi_n$ as
\[
	\op{osc}(S,\pi_n) := \max_{[t_j,t_{j+1}] \in \pi_n} \max_{r,s \in [t_j,t_{j+1}]} |S(s) - S(r)|.
\]
Here and in the following we write $[t_j,t_{j+1}] \in \pi_n$ to indicate that $t_j$ and $t_{j+1}$ are both in $\pi_n$ and are immediate successors (i.e. $t_j < t_{j+1}$ and $ \pi_n \cap(t_j , t_{j+1}) = \emptyset$).

\begin{definition}[$p$-th variation along a sequence of partitions]\label{def:p-var} Let $p>0$. A  continuous path $S \in C([0,T],\R)$ is said to have a $p$-th variation along a sequence of
partitions $\pi=(\pi_n)_{n\geq 1}$ if $\op{osc}(S,\pi_n)\to 0$ and the sequence of measures
	\[
		\mu^n \assign \sum_{[t_j, t_{j + 1}] \in \pi_n} \delta (\cdot - t_j) |S(t_{j + 1}) - S(t_j)|^p
	\]
	converges weakly to a measure $\mu$ without atoms. In that case we write $S \in V_p(\pi)$ and $[S]^p(t) := \mu([0,t])$ for $t \in [0,T]$, and we call $[S]^p$ the \emph{$p$-th variation} of $S$.
%
\end{definition}

\begin{remark}{\em \label{rem.pvariation}
\begin{enumerate}
	\item Functions in $V_p(\pi)$ do not necessarily have finite $p$-variation in the usual sense. Recall that the $p$-variation of a function $f \in C([0,T],\R)$ is defined as \cite{dudley2011}
		\[
			\|f \|_{\op{p-var}} \assign \Big(\sup_{\pi\in \Pi([0,T])} \sum_{[t_j, t_{j + 1}] \in \pi} |f(t_{j + 1}) - f(t_j)|^p\Big)^{1/p},
		\]
		where the supremum is taken over the set $\Pi([0,T])$ of all partitions $\pi$ of $[0,T]$. A typical example is the Brownian motion $B$, which has quadratic variation $[B]^2(t) = t$ along any refining sequence of partitions almost surely while at the same time having infinite 2-variation almost surely \cite{dudley2011,taylor1972}: $$\mathbb{P}\left( \|B\|_{\op{2-var}} = \infty\right) =1.$$  
	\item If $S \in V_p(\pi)$ and $q > p$, then $S \in V_q(\pi_n)$ with $[S]^q \equiv 0$.
\end{enumerate}}
\end{remark}
The following   characterization is given in \cite{cont2018pathwise}:
\begin{lemma} Let $S \in C([0,T],\R)$.
 $S \in V_p(\pi)$ if and only if there exists a continuous function $[S]^p$ such that
	\begin{equation}
	\forall t\in[0,T],\qquad \sum_{\substack{[t_j, t_{j + 1}] \in \pi_n: \\ t_j \le t}} |S(t_{j + 1}) - S(t_j)|^p\mathop{\longrightarrow}^{n\to\infty} [S]^p(t).\label{eq.pointwisecv}\end{equation}
 If this property holds, then the convergence in \eqref{eq.pointwisecv} is uniform.
\end{lemma}

\begin{example}\label{ex:fBm}
If $B$ is a fractional Brownian motion with Hurst index $H \in (0,1)$ and $\pi_n = \{kT/n: k \in \N_0\} \cap[0,T]$, then $B \in V_{1/H}(\pi)$ and $[B]^{1/H}(t) = t \E[|B_1|^{1/H}]$, see~\cite{pratelli2006,Rogers1997}.
\end{example}

Cont and Perkowski \cite{cont2018pathwise} obtained the following change of variable formula for $p \in \N$,   $S \in V_p(\pi)$ and $f \in C^p(\R,\R)$:
	\[
		f(S(t)) - f(S(0)) = \int_0^t f'(S(s))\dd S(s) + \frac{1}{p!} \int_0^t f^{(p)}(S(s)) \dd [S]^p(s),
	\]
	holds, where the integral 
	\[
		\int_0^t f'(S(s))\dd S(s) \assign \lim_{n \rightarrow \infty} \sum_{[t_j, t_{j + 1}] \in \pi_n} \sum_{k=1}^{p-1} \frac{f^{(k)} (S(t_j))}{k!}
   (S(t_{j + 1} \wedge t) - S(t_j \wedge t))^k
	\]
	is defined as a (pointwise) limit of compensated Riemann sums.

\begin{remark}[Relation with Young integration and rough integration]\em  
	Note however that, given the assumptions on $S$, the pathwise integral appearing in the formula cannot be defined as a Young integral. This relates to the observation in Remark \ref{rem.pvariation} that p-variation can be infinite for $S\in V_p(\pi)$.
	
	When $p=2$ it  reduces to an ordinary (left) Riemann sum. For $p>2$ such  compensated Riemann sums appear in the construction of `rough path integrals' ~\cite{FrizHairer,gubinelli2004}. Let $X \in C^\alpha([0,T],\R)$ be $\alpha$-H\"older continuous for some $\alpha \in (0,1)$, and write $q = \lfloor \alpha^{-1} \rfloor$. We can enhance $X$ uniquely into a (weakly) geometric rough path $(\bb X^1_{s,t}, \bb X^2_{s,t}, \dots, \bb X^q_{s,t})_{0 \le s \le t \le T}$, where $\bb X^k_{s,t} \assign  (X(t) - X(s))^k/k!$. Moreover, for $g \in C^{q+1}(\R,\R)$ the function $g'(X)$ is controlled by $X$ with Gubinelli derivatives
	\begin{align*}
		g'(X(t)) - g'(X(s)) & = \sum_{k=1}^{q-1} \frac{g^{(k+1)}(X(s))}{k!} (X(t) - X(s))^k + O(|t-s|^{q \alpha}) \\
		& = \sum_{k=1}^{q-1} g^{(k+1)}(X(s)) \bb X^k_{s,t} + O(|t-s|^{q \alpha}),
	\end{align*}
	and therefore the controlled rough path integral $\int_0^t g'(X(s)) \dd X(s)$ is given by
	\[
		\lim_{|\pi|\to 0} \sum_{[t_j, t_{j+1}] \in \pi} \sum_{k=1}^q g^{(k)}(X(s)) \bb X^k_{s,t} = \lim_{|\pi|\to 0} \sum_{[t_j, t_{j+1}] \in \pi} \sum_{k=1}^q g^{(k)}(X(s)) \frac{(X(t) - X(s))^k}{k!},
	\]
	where $|\pi|$ denotes the mesh size of the partition $\pi$, and which is the type of compensated Riemann sum   used to define the integral above, plus an additional term (Ito term) absorbed into the sum). 
\end{remark}
	

	\subsection{Fractional derivatives and fractional Taylor expansions}
	\label{sec.fractionalderivative}
	We recall some definitions on fractional derivatives and their properties. Several different notions of fractional derivatives exist in the literature and it is not clear which ones are the right tools for a given context. Our goal here is to shed some light on the advantages of different notions of fractional derivative.
	Much of this material may be found in the literature \cite{samko1987fractional}. 
	We have provided some detailed proofs for some useful properties whose proof we have not been able to find in the literature. 
	\begin{definition}[Riemann-Liouville fractional integration operator]
	    Suppose $f$ is a real function and $\alpha > 0$, the left Riemann-Liouville fractional integration operator of order $\alpha$ is defined by 
	    \[
	        I_{a^+}^\alpha f(x) := \frac{1}{\Gamma(\alpha)}\int_a^x(x-t)^{\alpha-1}f(t)dt
	    \]
	    if integration exists for $x > a \in \mathbb{R}$. Similarly, the right Riemann-Liouville fractional integration operator if given by 
	    \[
	        I_{b^-}^\alpha f(x) := \frac{1}{\Gamma(\alpha)}\int_x^b(t-x)^{\alpha-1}f(t)dt
	    \]
	    for $x < b \in \mathbb{R}$.
	\end{definition}
	This may be used to define a (non-local) fractional derivative of a real function $f$ associated with some base point, as follows:
	\begin{definition}[Riemann-Liouville fractional derivative]
	    Suppose $f$ is a real function and $n \leq \alpha < n+1$ for some integer $n \in \mathbb{N}$. Then left Riemann-Liouville fractional derivative of order $\alpha$ with base point $a$ at $x > a$ is defined by 
	    \[
	        D_{a^+}^\alpha f(x) := \left(\frac{d}{dx}\right)^{n+1} \left(I_{a^+}^{n+1-\alpha} f\right)(x)
	    \]
	    if it exists. Similarly, the right Riemann-Liouville fractional derivative of order $\alpha$ with base point $b$ and $b > x$ is given by 
	    \[
	        D_{b^-}^\alpha f(x) := \left(-\frac{d}{dx}\right)^{n+1} \left(I_{b^-}^{n+1-\alpha} f\right)(x)
	    \]
	\end{definition}
	Some basic properties of the Riemann-Liouville fractional derivative are described below.
	\begin{proposition}\label{RLFracDerivProp}
	    Suppose $f$ is a real function, $\alpha,\beta$ are two real numbers. We will also use the convention that $I_{a+}^\alpha = D_{a+}^{-\alpha}$. Then we have 
	    \begin{enumerate}
	        \item Fractional integration operators $\{I_{a+}^{\alpha},\alpha \geq 0\}$($\{I_{b-}^{\alpha},\alpha \geq 0\}$) form a semigroup in $L^p(a,b)$ for every $p\geq 1$. It is continuous in uniform topology for $\alpha > 0$ and strongly continuous for $\alpha\geq 0$.
	        \item $I_{a+}^{\alpha}\circ I_{a+}^\beta f= I_{a+}^{\alpha+\beta} f$($I_{b-}^{\alpha}\circ I_{b-}^\beta f= I_{b-}^{\alpha+\beta} f$) is valid in the following cases:
	        \begin{itemize}
	            \item $\beta \geq 0, \alpha+\beta \geq 0, f \in L^1([a,b])$
	            \item $\beta \leq 0, \alpha \geq 0, f \in I_{a+}^{-\beta}(L^1)(I_{b-}^{-\beta}(L^1))$
	            \item $\alpha \leq 0, \alpha+\beta \leq 0, f \in I_{a+}^{-\alpha-\beta}(L_1)(I_{b-}^{-\alpha-\beta}(L_1))$
	        \end{itemize}
	        \item If $n \leq \alpha < n+1$ and $D_{a^+}^\alpha f \in L^1([a,y])$ for some $y > x \in [a,b]$ , then we have 
	        \[
	            I_{a^+}^\alpha D_{a^+}^\alpha f(x) = f(x) - \sum_{j=1}^{n+1} \frac{D_{a^+}^{\alpha - j}f(a)}{\Gamma(\alpha - j + 1)}\left(x-a\right)^{\alpha - j}
	        \]
	        and similarly,
	        \[
	            I_{b^-}^\alpha D_{b^-}^\alpha f(x) = f(x) - \sum_{j=1}^{n+1} \frac{(-1)^j D_{b^-}^{\alpha - j}f(b)}{\Gamma(\alpha - j + 1)}\left(b-x\right)^{\alpha - j}
	        \]
	        Here $D^{\kappa}_{a^+}f(a) = \lim_{y\rightarrow a,y > a}D^{\kappa}_{a^+}f(y)$ for any $\kappa \in \mathbb{R}$ and similar for the right Riemann-Liouville fractional derivative.
	    \end{enumerate}
	\end{proposition}
	The Riemann-Liouville derivative has several shortcomings. One of them is that the fractional derivative of a constant is not zero. To overcome this, we introduce a modification of Riemann-Liouville fractional derivative which is called the {\it Caputo derivative}.
	\begin{definition}[Caputo derivative]
	    Suppose $f$ is a real function and $n+1 \geq \alpha > n$. We define the left and right Caputo fractional derivatives of order $\alpha$ at $x \in (a,b)$ by 
	    \[
	        C_{a^+}^\alpha f(x) = D_{a^+}^\alpha \left[f(t) - \sum_{k=0}^{n}\frac{f^{(k)}(a)}{k!}(t-a)^k\right](x)
	    \]
	    \[
	        C_{b^-}^\alpha f(x) = D_{b^-}^\alpha \left[f(t) - \sum_{k=0}^{n}\frac{(-1)^k f^{(k)}(b)}{k!}(b-t)^k\right](x)	    
	    \]
	    here $f^{(k)}$ denotes the $k$-th derivative of $f$.
	\end{definition}
	We enumerate below some useful properties of the Caputo derivative:
	\begin{proposition}\label{Caputoprop}
	Suppose $f$ is a real function and $\alpha,\beta>0$.
	    \begin{enumerate}
	        \item If $\alpha = n \in \mathbb{N}$ and $f$ is $n-$th differentiable, then $C_{a^+}^\alpha f = f^{(n)}$ and $C_{b^-}^\alpha f = (-1)^n f^{(n)}$
	        \item Let $0 < \alpha \leq \beta $ and $f \in L^1[a,b]$, then we have $C_{a^+}^\alpha I_{a^+}^\beta f = I_{a+}^{\beta - \alpha}f $ and $C_{b^-}^\alpha I_{b^-}^\beta f = I_{b-}^{\beta - \alpha}f$
	        \item $C_{a^+}^\alpha C_{a^+}^n = C_{a^+}^{\alpha+n}$ and $C_{b^-}^\alpha C_{b^-}^n = C_{b^-}^{\alpha + n}$ for all $n\in\mathbbm{N}$.
	    \end{enumerate}
	\end{proposition}
	\begin{example}\label{xalpha}
	    We give some simple examples of Caputo fractional derivative here. 
	    \begin{enumerate}
	        \item Consider $f(x) = |x|^\alpha$ and $0 < \beta \leq \alpha < 1$. Then we have 
	        \begin{eqnarray*}
	            C_{a^+}^\beta f(x) &=& \frac{d}{dx}\left(\frac{1}{\Gamma(1-\beta)}\int_a^x\frac{|t|^\alpha - |a|^\alpha}{(x-t)^\beta}dt\right) \\
	            &=& \frac{1}{\Gamma(1-\beta)}\frac{d}{dx}\left(\int_a^x|t|^\alpha(x-t)^{-\beta}dt - |a|^\alpha\frac{(x-a)^{1-\beta}}{1-\beta}\right) \\
	            &=& \frac{1}{\Gamma(1-\beta)}\frac{d}{dx}\left|\int_{\frac{a}{x}}^1 |t|^\alpha|1-t|^{-\beta}dt |x|^{\alpha-\beta+1}\right| - \frac{1}{\Gamma(1-\beta)}|a|^\alpha(x-a)^{-\beta} \\
	            &=& \frac{1}{\Gamma(1-\beta)}\left(|a|^\alpha(x-a)^{-\beta}\frac{a-x}{x} + (\alpha - \beta + 1)|x|^{\alpha - \beta}\int_{\frac{a}{x}}^1|t|^\alpha|1-t|^{-\beta}dt\right)\\
	        \end{eqnarray*}
	        So we can see directly that $f^{(\beta+)}(a) = 0$ for any $a \neq 0$ or $\beta < \alpha$ but $f^{(\alpha+)}(0) = \Gamma(\alpha + 1)$. It can be seen further that for all $a\geq 0$, $C_{a^+}^\alpha f$ is continuous on $[a,\infty]$ but for $a < 0$, $C_{a^+}^\alpha f$ has singularity at point $0$. In particular, for $\beta = \alpha$, we have 
	        \[
	            C_{a^+}^\alpha f(x) = \frac{1}{\Gamma(1-\beta)}\int_{\frac{a}{x}}^1 g(t) - g\left(\frac{a}{x}\right) dt
	        \]
	        with $g(t) = \left|\frac{t}{1-t}\right|^\alpha$
	    \end{enumerate}
	\end{example}
	Associated with the Caputo derivative is a fractional Taylor expansion:
	\begin{proposition}\label{TaylorExpansionFrac}
	    Let $f \in C^n([a,b])$ and $n+1 \geq \alpha > n$. If the Caputo fractional derivative of order $\alpha$ of $f$ exists on $[a,b]$ and $C_{a^+}^\alpha f \in L^1([a,b])$, then we have the Taylor expansion:
	    \[
	        f(x) = \sum_{k=0}^{n} f^{(k)}(a)\frac{(x-a)^k}{k!} + \frac{1}{\Gamma(\alpha)}\int_{a}^x\frac{C_{a^+}^\alpha f(t)}{(x-t)^{1-\alpha}}dt
	    \]
	    and 
	    \[
	        f(x) = \sum_{k=0}^{n} f^{(k)}(b)\frac{(x-b)^k}{k!} + \frac{1}{\Gamma(\alpha)}\int_x^b\frac{C_{b^-}^\alpha f(t)}{(t-x)^{1-\alpha}}dt
	    \]
	\end{proposition}
	As we were not able to find a  proof of this expansion in the literature,  we provide a detailed proof below. Let us start with a simple but useful lemma, which corresponds to Theorem 3.3 in  \cite{jarad2017generalized} for the case $\rho = 1$:
	\begin{lemma}\label{lemCaputo}
	 Suppose $f \in C^n([a,b])$, then for any $0 < \alpha < n$ and $\alpha \notin \mathbb{N}$, we have actually $C_{a^+}^\alpha f(a) = C_{b^-}^\alpha f(b) = 0$.
	\end{lemma}
	\begin{proof}[Proof of Proposition \ref{TaylorExpansionFrac}.]
    Define 
	\[
	    g(x) = f(x) - \sum_{k=0}^n \frac{f^{(k)}(a)}{k!}(x-a)^k.
	\]
	From proposition \ref{RLFracDerivProp}, we have 
	\[
	    I_{a^+}^\alpha D_{a^+}^\alpha g(x) = g(x) - \sum_{j=1}^{n+1} \frac{D_{a^+}^{\alpha - j}g(a)}{\Gamma(\alpha - j + 1)}\left(x-a\right)^{\alpha - j}
	\]
	Now we have by linearity of Riemann-Liouville fractional derivative operator and definition of Caputo fractional derivative that
	\[
	    D_{a^+}^{\alpha - j}g(a) = C_{a^+}^{\alpha - j}f(a) - \sum_{k = n-j+1}^n\frac{f^{(k)}(a)}{k!}D_{a^+}^{\alpha - j}((x-a)^k)(a)
	\]
	Furthermore, for $k \geq n-j+1 > \alpha - j$, $D_{a^+}^{\alpha - j}((x-a)^k)(a) = C_{a^+}^{\alpha - j}((x-a)^k)(a)$ by definition. Hence we have 
	\[
		D_{a^+}^{\alpha - j}g(a) = C_{a^+}^{\alpha - j}f(a) - \sum_{k = n-j+1}^n\frac{f^{(k)}(a)}{k!}C_{a^+}^{\alpha - j}((x-a)^k)(a)
	\]
	Now since $f \in C^n$ and $(\cdot-a)^k \in C^{\infty}$, we see that actually $D_{a^+}^{\alpha - j}g(a) = 0$  for every $j = 1, \cdots, n$ by lemma \ref{lemCaputo}. For the case $j = n + 1$, we actually have that 
	\begin{eqnarray*}
	    \left|D_{a^+}^{\alpha-n-1}g(a)\right| &=& \lim_{y \rightarrow a, y > a} \left|I_{a^+}^{n+1-\alpha}g(y)\right|   \\
	    &=& \lim_{y\rightarrow a, y > a} \frac{1}{\Gamma(n+1-\alpha)}\left|\int_a^y \frac{g(t)}{(y-t)^{\alpha - n}}dt\right| \\
	    &\leq& \lim_{y\rightarrow a, y > a} \frac{1}{\Gamma(n+1-\alpha)}\left|\int_a^y \frac{\max_{t \in [a,y]}\{|g(t)|\}}{(y-t)^{\alpha - n}}dt\right| \\
	    &=& \lim_{y\rightarrow a, y > a} \frac{1}{\Gamma(n+1-\alpha)}\max_{t \in [a,y]}\{|g(t)|\}(y-a)^{n+1-\alpha} = 0
	\end{eqnarray*}
	So actually we have 
	\[
	    g(x) = I_{a^+}^\alpha D_{a^+}^\alpha g(x) = I_{a^+}^\alpha C_{a^+}^\alpha f(x),
	\]
	hence the result. The derivation of the  other formula is similar. 
	\end{proof}
	The above derivative operators are non-local operators.
	We now introduce the concept of {\it local} fractional derivative:
	\begin{definition}[Local fractional derivative]\label{def.localderivative}
	    Suppose $f$ is left(resp. right) fractional differentiable of order $\alpha$ on $[a,a+\delta]([a-\delta,a])$ for some positive $\delta$, then the left(resp. right) local fractional derivative of order $\alpha$ of function $f$ at point $a$ is given by 
	    \[
	        f^{(\alpha+(resp. \alpha-))}(a) = \lim_{y\rightarrow a, y\geq(\leq) a}C_{a+(a-)}^\alpha(y)
	    \]
	\end{definition}
	\begin{remark}
	    Using the integration by parts formula, we can see that we actually have 
	    \[
	        D_{a^+}^\alpha f(x)= \sum_{k=0}^{n-1} \frac{f^{(k)}(a)(x-a)^{k-\alpha-n}}{\Gamma(k+2-\alpha)} + D_{a^+}^{(\alpha-n)} f^{(n)}(x)
	    \]
	    Hence we will have furthermore that $C_{a^+}^\alpha f(x) = C_{a^+}^{(\alpha - n)}f^{(n)}(x)$. This means the existence of one side will imply the existence of the another side. And if we take  the limit we can see that the existences of $f^{(\alpha+)}(x)$ and $(f^{(n)})^{(\alpha+)}(x)$ are equivalent and they are equal. This is very important in the following proofs.
	\end{remark}
	\begin{corollary}\label{fracpiano}
	    Let $f \in C^n([a,b])$ and $n+1 \geq \alpha > n$. Suppose left(right) local fractional derivative of order $\alpha$ of $f$ exists at $a$. This leads to the following fractional Taylor formula with  remainder term for $x\in[a,b]$
	    \begin{eqnarray}
	        f(x) &= \sum_{k=0}^{n} f^{(k)}(a)\frac{(x-a)^k}{k!} + \frac{1}{\Gamma(\alpha + 1)}f^{(\alpha+)}(a)(x-a)^\alpha + o(|x-a|^\alpha) \label{eq.piano1}\\
	        f(x) &= \sum_{k=0}^{n} f^{(k)}(b)\frac{(x-b)^k}{k!} + \frac{1}{\Gamma(\alpha + 1)}f^{(\alpha-)}(b)(b-x)^\alpha + o(|x-b|^\alpha)\label{eq.piano2}
	     \end{eqnarray}
	\end{corollary}
	\begin{proof}
	    We only prove for the left Caputo derivative case. 
	    \begin{eqnarray*}
	        \frac{1}{\Gamma(\alpha)}\int_a^x \frac{C_{a^+}^\alpha f(t)}{(x-t)^{1-\alpha}}dt &=& \frac{1}{\Gamma(\alpha)}\int_a^x \frac{f^{(\alpha+)}(a)}{(x-t)^{1-\alpha}}dt + \frac{1}{\Gamma(\alpha)}\int_a^x \frac{C_{a^+}^\alpha f(t) - f^{(\alpha+)}(a)}{(x-t)^{1-\alpha}}dt \\
	        &=& \frac{f^{(\alpha+)}(a)}{\Gamma(\alpha + 1)}(x-a)^\alpha + \frac{1}{\Gamma(\alpha)}\int_a^x \frac{C_{a^+}^\alpha f(t) - f^{(\alpha+)}(a)}{(x-t)^{1-\alpha}}dt
	    \end{eqnarray*}
	    Furthermore,
	    \begin{eqnarray*}
	        \left|\frac{1}{\Gamma(\alpha)}\int_a^x \frac{C_{a^+}^\alpha f(t) - f^{(\alpha+)}(a)}{(x-t)^{1-\alpha}}dt\right| &\leq&\frac{1}{\Gamma(\alpha)} \int_a^x \frac{\max_{t\in[a,x]}\{|C_{a^+}^\alpha f(t) - f^{(\alpha+)}(a)|\}}{(x-t)^{1-\alpha}}dt \\
	        &=& \frac{\max_{t\in[a,x]}\{|C_{a^+}^\alpha f(t) - f^{(\alpha+)}(a)|\}}{\Gamma(\alpha + 1)} (x-a)^\alpha
	    \end{eqnarray*}
	    Hence 
	    \[
	        \lim_{x\rightarrow a, x\geq a} \frac{ \left|\frac{1}{\Gamma(\alpha)}\int_a^x \frac{C_{a^+}^\alpha f(t) - f^{(\alpha+)}(a)}{(x-t)^{1-\alpha}}dt\right|}{(x-a)^\alpha} \leq \lim_{x\rightarrow a, x\geq a} \frac{\max_{t\in[a,x]}\{|C_{a^+}^\alpha f(t) - f^{(\alpha+)}(a)|\}}{\Gamma(\alpha + 1)} = 0
	    \]
	    Thus we proved the result.
	\end{proof}
	A similar mean value theorem holds for the non-local fractional derivative. The following result, which we state for completeness, is a direct consequence of proposition \ref{TaylorExpansionFrac}:	\begin{corollary}
	    Let $f \in C^n([a,b])$ and $n+1 \geq \alpha > n$. Suppose Caputo fractional derivative of order $\alpha$ of $f$ is continuous on $[a,b]$. Then we have the for $x \in [a,b]$, there exists $\xi \in [a,x]$ such that
	    \[
	        f(x) = \sum_{k=0}^{n} f^{(k)}(a)\frac{(x-a)^k}{k!} + \frac{1}{\Gamma(\alpha + 1)}C_{a^+}^\alpha f(\xi)(x-a)^\alpha
	    \]
	    A similar formula holds for the Caputo  right derivative.
	\end{corollary}

	\subsection{A characterization of the local fractional derivative}
	\begin{proposition}\label{almostsurelyzero}
	    Let $p \notin\mathbb{N}$ and  $f \in C^p([a,b])$. If $f^{(p+)}$ exists everywhere on $[a,b]$ then $f^{(p+)} = 0$ almost everywhere on $[a,b]$.
	\end{proposition}
	A rather complex proof of this proposition is given in \cite[Corollary 3]{chen2010local}. We here give a simple proof of this property using only properties of monotone functions.
	\begin{proof}
	    We first suppose $0 < p < 1$. And by Taylor expansion theorem for fractional derivative, we actually have
	    \[
	        f(y) = f(x) + g(x)(y-x)^p + o(|y-x|^p)
	    \]
	    Hence for any sequence $y_n \rightarrow x, y_n > x$, we have 
	    \[
	        \lim_{n\rightarrow\infty} \frac{f(y_n)-f(x)}{(y_n - x)^p} = g(x)
	    \]
	    Now let $E^+ = \{x \in\mathbb{R}: g(x) > 0\}$. If $\mathcal{L}(E^+) > 0$, then $\exists K$ compact set such that $K \in E^+$ and $\mathcal{L}(K) > 0$. Here $\mathcal{L}$ denotes Lebesgue measure on the real line. 
	    
	    Let's consider the open set $\cup_{x \in K}(x, x+\delta_x)$ in $\mathbb{R}$, where $\delta_x$ is the largest number so that $f(y) > f(x),\forall y \in (x,x+\delta_x)$. We can then write $\cup_{x \in K}(x, x+\delta_x) = \cup_{k=1}^\infty I_k$ for some open interval $I_k$. By the additive property of measure, $\exists I_r$ such that $\mathcal{L}(K\cap\bar{I_r}) > 0$. 
	    
	    For any $\bar{x}<\bar{y} \in K\cap\bar{I_r}$, there exists $x_0 \in K\cap [\bar{x},\bar{y}]$ such that 
	    \[
	        x_0 = \argmax_{x \in K\cap[\bar{x},\bar{y}]}\{f(x)\}
	    \]
	    If $x_0\neq \bar{y}$, then there exists $y_0 \in (x_0,x_0+\delta_{x_0})\cap K\cap[\bar{x},\bar{y}]$ such that $f(y_0) > f(x_0)$, which is a contradiction. Hence $x_0 = \bar{y}$ and we have $f(\bar{y}) > f(\bar{x}),\forall \bar{x} < \bar{y} \in K\cap\bar{I_r}$.
	    
	    Now define $\bar{f}: \bar{I_r}\rightarrow \mathbb{R}$ such that $\bar{f} = f$ on $K\cap\bar{I_r}$ and $\bar{f}$ is linear outside of $K\cap\bar{I_r}$. In fact, let $\tilde{x} \in \bar{I_r}$, if $\tilde{x} \notin K$, then $\tilde{x} \in (z_0,z_1)$ with $z_0,z_1 \in K$. By above reason, $\bar{f}(z_1) \geq \bar{f}(z_0)$, and we perform linear interpolation to define the value $\bar{f}(\tilde{x})$ here. It is then easy to conclude that $\bar{f}$ is an increasing function. Hence $\bar{f}$ is differentiable almost everywhere on $\bar{I_r}$.
	    
	    Now we go back to the function $f$ on set $K$. It could be concluded that almost surely on $K$, $\bar{f}$ is differentiable and we can simply suppose here $\bar{f}$ is differentiable and $f = \bar{f}$ on $K$ whose Lebesgue measure is positive. Now choose a point $x_1 \in K$, either $x_1$ is isolated in $K$ or it is an accumulated point in $K$. For the latter case, we would have 
	    \[
	        g(x_1) =  \lim_{n\rightarrow\infty} \frac{f(y_n)-f(x_1)}{(y_n - x)^p} = \lim_{n\rightarrow\infty} \frac{\bar{f}(y_n)-\bar{f}(x_1)}{(y_n - x_1)^p} = 0
	    \]
	    Here $\{y_n\}$ is a sequence in $K$ approximating $x_1$. Since the set of all isolated points in $K$ is of zero measure,
	   we could get $g = 0$ almost surely on $K$, which is a contradiction. Hence $\mathcal{L}(E^+) = 0$. Similarly, $\mathcal{L}(E^-) = 0$. Hence $f^{(p+)}(x) = 0$ almost surely on $\mathbb{R}$.
	    
	    Now if $m<p<m+1$ for some integer $m > 0$, since we have $f^{(p+)} = (f^{(m)})^{(p-m+)}$, we could conclude again that $f^{(p+)}(x) = 0$ almost surely on $\mathbb{R}$.
	\end{proof}
	\begin{remark}
	    Here actually we only need a weaker condition on the function, namely that 
	    \[
	        \lim_{y\rightarrow x, y > x} \frac{f^{(m)}(y)-f^{(m)}(x)}{|y-x|^\alpha}
	    \]
	    exists for $x\in \mathbb{R}$ with $m = \lfloor p \rfloor$ and $p = m + \alpha$, which could be thought as classical definition of $\alpha$ order derivative. We call 
	    \[
	        \lim_{x\rightarrow a, x > a}\frac{f^{(n)}(x) - f^{(n)}(a)}{|x-a|^\alpha}
	    \]
	    left 'classical' fractional derivative of order $\alpha$ if it exists and similar for right derivatives. However, when this limits exists, it is not guaranteed that local fractional derivative of $f$ exists. 
	\end{remark}
	    We can actually give a stronger result. Let $E = \{x\in\mathbb{R}: \lim_{y\rightarrow x, y > x} \frac{f^{(m)}(y)-f^{(m)}(x)}{|y-x|^\alpha} \neq 0\}$. 
	 \begin{proposition}
	     The Hausdorff dimension of the set $E$ is at most $\alpha$.
	 \end{proposition}
	 \begin{proof}
	     For each $\epsilon, \delta$, we define the set $E_{\epsilon+}^\delta$ to be the subset of $E$ such that for all $x \in E_{\epsilon+}^\delta$, we have
	     \[
	        f(y) > f(x) + \epsilon(y-x)^\alpha, \forall x<y<x+\delta
	     \]
	     Then by the countably additive of measure, we only need to show the Hausdorff dimension of the set $E_{\epsilon+}^\delta$ is zero. Furthermore, we can restrict the set $E_{\epsilon+}^\delta$ on the interval $[0,\delta]$ due to the countably additive of measure. Hence we will work on $[0,\delta]$ with $E_{\epsilon+}^\delta$. 
	     
	     First of all, suppose the Hausdorff dimension of set $E_{\epsilon+}^\delta$ is larger than $\alpha$. In this case, we considers partitions $\{0,\frac{\delta}{k},\cdots,\frac{(k-1)\delta}{k},\delta\}$ for $k\in\mathbb{N}$. In each interval $[\frac{i\delta}{k},\frac{(i+1)\delta}{k}]$, we choose the leftmost and rightmost points of the set $E_{\epsilon+}^\delta$ and denote this interval as $I_k^i$(if exists). Then $\cup_{i=0}^{k-1} I_k^i$ covers the set $E_{\epsilon+}^\delta$ and by the definition of Hausdorff measure, we know that $\sum_{i=0}^{k-1}|I_k^i|^\alpha \rightarrow \infty$ as $k\rightarrow \infty$. Now let $a_k$ be the left boundary point of $\cup_{i=0}^{k-1} I_k^i$ and $b_k$ be the right boundary point of $\cup_{i=0}^{k-1} I_k^i$. We know 
	     \[
	        f(b_k) - f(a_k) > \epsilon \sum_{i=0}^{k-1}|I_k^i|^\alpha \rightarrow \infty
	     \]
	     which gives a contradiction when $k\rightarrow\infty$.
	 \end{proof}
	
	 The following result is a  consequence of the fractional Taylor expansion \eqref{eq.piano1}-\eqref{eq.piano2}:
	\begin{proposition}
	    Suppose $f \in C^\alpha([a,b])$, then for any $\beta < \alpha$ and $\beta \notin\mathbb{N}$, we have $f^{(\beta+)}(x) = 0$ if $f^{(\beta+)}(x)$ exists. And for $\beta > \alpha$ with same integer part, $f^{(\beta+)}(x)$ will  either be $\pm\infty$ or not exist.
	\end{proposition}


	\newpage\section{Fractional It\^o Calculus}
\label{sec.fractionalIto}

	In this section we  derive  It\^o-type change of variable formulas for smooth functions of paths with  $p-$th variation along a sequence of partitions $\{\pi_n\}_{n\in\mathbb{N}}$, first for functions (Sections \ref{sec.timeindependent} and \ref{sec.timedependent}) then for path-dependent functionals (Section \ref{sec.pathdependent}). In each case the focus is the existence or not of a 'fractional' It\^o remainder term: we will see that the existence of a non-zero It\^o term depends on the fine structure of the function and its fractional derivative.
	
	\subsection{A change of variable formula  for time-independent functions}\label{sec.timeindependent}
	We now derive an It\^o-type change of variable formula for smooth functions of paths with  $p-$th variation along a sequence of partitions $\{\pi_n\}_{n\in\mathbb{N}}$. 
	
	Let $S\in V_p(\pi)$ be a path which admits $p-$th order variation along some sequence of partitions $\{\pi_n\}_{n\in\mathbb{N}}$. We make the following assumptions:
	\begin{assumption}\label{ass1}
	    For any $k\in\mathbb{R}$, we have 
	    \[
	        \int_0^T \1_{\{S(t)=k\}}d[S]_\pi^p = 0
	    \]
	\end{assumption}
      Let $m = \lfloor p \rfloor$.
	\begin{assumption}\label{ass3}
	   $f\in C^m(\mathbb{R})$ and  admits a left local fractional derivative of order $p$ everywhere.
	\end{assumption}

	\begin{assumption}\label{ass4}
	    The set 	\[
	    \Gamma_f = \{x \in \mathbb{R}: \exists U \ni x \text{ open, } C_{a^+}^p f(b) \text{ is continuous on }\{(a,b) \in U\times U: a \leq b\} \}
	\] is locally finite, i.e. for any compact set $K \in \mathbb{R}$, the set $\Gamma_f^c \cap K$ has only finite number of points.
	\end{assumption}

	    Assumption $\ref{ass1}$ will be automatically satisfied if the occupation measure  of  $S$ is atom-less. In fact, 
	    \[
	        \Gamma_S(\{k\}) = \int_0^T \1_{\{S(t) = k\}} d[S]_\pi(t)
	    \]
	    	    We first give a simple lemma regarding  the set $\Gamma_f$:
	    \begin{lemma}\label{lemGamma_f} 
	        $\forall x\in \Gamma_f$, the left local fractional derivative of order $p$ of $f$ exists and equals to zero, i.e.  $f^{(p+)}(x) = 0$
	    \end{lemma}
	    \begin{proof}
	       Let $x \in \Gamma_f$. There exists $x \in U$ open such that $C_{a+}^p f(b)$ is continuous on $\{(a,b)\in U\times U: a\leq b\}$. Hence the function $f^{(p+)}(x)$ is continuous on $U$. Due to the fact that $f^{(p+)}$ is zero almost surely, we know $f^{(p+)}(x) = 0$.
	    \end{proof}
	    
	    Assumption $\ref{ass1}$ will be satisfied in particular if $S$ admits an occupation density \cite{geman1980},  then . However, as the following example shows, Assumption \ref{ass1} may fail to be satisfied even if the path  has a non-zero $p-$th order variation. 
	   
	\begin{remark}[A counterexample] \label{counterass1}
	    We now give an example of path failing to satisfy Assumption \ref{ass1}, in the spirit of \cite[Example 3.6.]{davis2018}.\footnote{  \cite[Example 3.6]{davis2018} aims to construct a path which admits quadratic variation but does not possess local time along some sequence of partitions. In fact there seems to be an issue with the construction in \cite{davis2018} but the underlying idea is still useful for our construction.}
	    
	    Let $p > 2$ and define the intervals
	    \begin{eqnarray}
	        I_1^1 = \left(\frac{1}{3},\frac{2}{3}\right), \quad I_2^1 = \left(\frac{1}{9},\frac{2}{9}\right), \quad I_2^2 = \left(\frac{7}{9},\frac{8}{9}\right), \cdots, I_j^i, j = 1,\cdots,2^{i-1}\label{eq.Iji}
	    \end{eqnarray}
	    and let $C = [0,1]\slash \cup_{i=1}^\infty\cup_{j=1}^{2^{i-1}} I^i_j$, which is the Cantor ternary set \cite{cantor1884}. Let $c:[0,1]\to \mathbb{R}_+$ be the associated Cantor function, which is defined by
	    \[
	        c(x) = \left\{
	        \begin{array}{ll}
	             \sum_{n=1}^{\infty} \frac{a_n}{2^n}, & x = \sum_{n=1}^{\infty} \frac{2a_n}{3^n} \in C, \quad a_n \in \{0,1\}  \\
	             \sup_{y\leq x, y \in C} c(y),& x \in [0,1]\slash C
	        \end{array}
	        \right.
	    \]
	    We can see it is a non-decreasing function increasing only on Cantor set $C$.
	    
	    Consider the function 
	    \[
	        S(t) = |2^{\log_3(2\cdot \min_{u\in C}|t-u|)}|^{\frac{1}{p}}.
	    \]
	    We are going to construct a sequence of partitions such that the $p-$th variation of $S$ along this sequence will be the Cantor function $c$. In this case, we will see 
	    \[
	        \int_0^1 \1_{\{S(t)=0\}}d[S]^p = \int_0^1 \1_{t\in C}dc(t) = c(1) - c(0) = 1
	    \]
	    which shows Assumption \ref{ass1} is not satisfied.
	    
	    We begin with the partition $\{\pi_{j,n}^i\}$, which denotes the $n-$th partitions in the interval $I_j^i$. Define $t_{j,n}^{i,0} = \inf I_j^i$ and 
	    \[
	        t_{j,n}^{i,k+1} = \inf\left\{t > t_{j,n}^{i,k}: S(t) \in \left(\frac{1}{k_n}\sup_{t \in I_j^i} S(t)\right)\mathbb{Z}\right\}
	    \]
	    then $t_{j,n}^{i,2k_n} = \sup I_j^i$. $k_n$ is an integer to be determined. We then do the calculation 
	    \[
	    \sum_{k=0}^{2k_n - 1}|S(t_{j,n}^{i,k+1}) - S(t_{j,n}^{i,k})|^p = \sum_{k=0}^{2k_n-1}\left|\frac{1}{k_n}2^{-\frac{i}{p}}\right|^p = 2^{1-i}k_n^{1-p}
	    \]
	    Then the sum for the $n-$th partition will be 
	    \[
	        \sum_{i=1}^n\sum_{j=1}^{2^{i-1}}2^{1-i}k_n^{1-p} = n\cdot k_n^{1-p}
	    \]
	    We choose $k_n = \lfloor n^{\frac{1}{p-1}} \rfloor $. Then 
	    \[
	        nk_n^{1-p} \geq n\cdot(n^{\frac{1}{p-1}})^{1-p} = 1
	    \]
	    and 
	    \[
	        nk_n^{1-p} \leq n\cdot(n^{\frac{1}{p-1}} - 1)^{1-p} = \left(1 - n^{\frac{1}{1-p}}\right)^{1-p} \rightarrow 1
	    \]
	    as $n\rightarrow\infty$. Hence we see 
	    \[
	        [S]_\pi^p(1) = 1 
	    \]
	    Furthermore, since 
	    \[
	        2^{1-i}k_n^{1-p} \rightarrow 0
	    \]
	    we see that $[S]_\pi^p$ will not change on the interval $I_j^i$ and hence due to symmetry, we finally can show that $[S]^p(t) = c(t)$, the Cantor function.
	\end{remark}
	\begin{remark}\label{rem.ass4}
	    Assumption \ref{ass4} is necessary for Theorem \ref{itoformula1}. There are continuous function with $f^{(p+)} = 0$ everywhere but no joint continuity for the Caputo derivative. We note here that the joint continuity can lead to the uniform convergence of the equation
	    \[
	        \lim_{y\rightarrow x}\frac{f(y) - f(x)}{|y-x|^\alpha} = 0
	    \]
	    for $x \in K$ with compact set $K$.
	\end{remark}
    The following theorem extends the result of \cite{cont2018pathwise} to the case of functions with fractional regularity:
    \begin{theorem}\label{itoformula1}
       Let $S\in V_p(\pi)$ satisfying assumptions \ref{ass1}. If $f$ is a continuous function satisfying assumptions $\ref{ass3}$ and $\ref{ass4}$ for $m = \lfloor p \rfloor$, then  
        \[
	     \forall t\in [0,T],\qquad   f(S(t)) = f(S(0)) + \int_0^T (f'\circ S).dS
	    \]
	    where the last term is a limit of compensated Riemann sums of order $m = \lfloor p \rfloor$:
	    \[
	       \int_0^T (f'\circ S).dS = \lim_{n\rightarrow\infty}\sum_{t_i\in\pi_n} \sum_{j=1}^m \frac{f^{(j)}(S(t\wedge t_i))}{j!}(S(t\wedge t_{j+1}) - S(t\wedge t_j))^j.
	    \]
    \end{theorem}
    \begin{proof}[Proof of Theorem \ref{itoformula1}]
        We first suppose that for any $t_i \in \pi_n$ other than $0$ and $T$, $f^{(p+)}(S_{t_i}) = 0$. By Taylor's formula with integral remainder, we have 
        \begin{eqnarray*}
	    &&f(S(t)) - f(S(0)) = \sum_{[t_i,t_{i+1}]\in\pi_n} f(S(t\wedge t_{i+1})) - f(S(t\wedge t_i)) \\
	    &=&  \sum_{[t_i,t_{i+1}]\in\pi_n}\left(f(S(t\wedge t_{i+1})) - f(S(t\wedge t_i)) - \cdots - \frac{f^{(n)}(S(t\wedge t_i))}{n!}(S(t\wedge t_{i+1}) - S(t\wedge t_i))^n\right) \\
	    &+& \sum_{[t_i,t_{i+1}]\in\pi_n} \sum_{j=1}^m \frac{f^{(j)}(S(t\wedge t_i))}{j!}(S(t\wedge t_{i+1}) - S(t\wedge t_i))^j \\
	    &=& L^n + \sum_{[t_i,t_{i+1}]\in\pi_n} \frac{1}{\Gamma(m)}\int_{S(t\wedge t_i)}^{S(t\wedge t_{i+1})}(f^{(m)}(r) - f^{(m)}(S(t\wedge t_i)))(S(t\wedge t_{i+1}) - r)^{m-1}dr
	\end{eqnarray*}
	where $L^n = \sum_{[t_i,t_{i+1}]\in\pi_n} \sum_{j=1}^m \frac{f^{(j)}(S(t\wedge t_i))}{j!}(S(t\wedge t_{i+1}) - S(t\wedge t_i))^j$. It is easy to notice that there exists a constant $M>0$ and 
	\[
	    \left|\frac{1}{\Gamma(m)}\int_{S(t\wedge t_i)}^{S(t\wedge t_{i+1})}(f^{(m)}(r) - f^{(m)}(S(t\wedge t_i)))(S(t\wedge t_{i+1}) - r)^{m-1}dr\right| \leq M|S(t\wedge t_{i+1}) - S(t\wedge t_i)|^p
	\]
	for example $M = \frac{\Gamma(p-m+1)||f||_{C^p}}{\Gamma(p+1)}$. We denote $C^n_\epsilon = \{[t_i,t_{i+1}]\in\pi_n: |S(t_i) - k|\leq \epsilon \text{ for some }k \in \Gamma_f^c\}$. Since $S\in V^p(\pi)$, there exists $N_\epsilon \in \mathbb{N}$ such that for all $n > N_\epsilon$, we have for all $[t_i,t_{i+1}]\in\pi_n$
	\[
	    |S(t_{i+1}) - S(t_i)| \leq \frac{\epsilon}{2}
	\] 
	Thus for $[t_i,t_{i+1}] \notin C^n_\epsilon$, we have 
	\[
	    (S(t_{i+1}) - k)(S(t_i) - k) > 0
	\]
	for all $k \in \Gamma_f^c$ and $S(t_i), S(t_{i+1})$ is away from $\Gamma_f^c$ with at least $\frac{\epsilon}{2}$ distance, which means we have the estimation
	\begin{equation}\label{continuity}
	    |C_{a_1^+}^p f(b_1) - C_{a_2^+}^p f(b_2)| \leq \omega_{\epsilon}(\sqrt{(a_1 - a_2)^2 + (b_1 - b_2)^2})
	\end{equation}
	for some modulus of continuity $\omega_{\epsilon}$, $a_1,a_2,b_1,b_2 \in [S(t_i),S(t_{i+1})]$ or $[S(t_{i+1}),S(t_i)]$ and $a_1 \leq b_1, a_2 \leq b_2$ by assumption $\ref{ass3}$ and $\ref{ass4}$. Furthermore, we denote $C_{1,\epsilon}^n = \{[t_i,t_{i+1}] \in \pi_n\slash C_\epsilon^n: S(t_i) \leq S(t_{i+1}), t_i \neq 0, t_{i+1}\neq T\}$ and $C_{2,\epsilon}^n = \{[t_i,t_{i+1}] \in \pi_n\slash C_\epsilon^n: S(t_{i+1}) < S(t_i), t_i \neq 0, t_{i+1}\neq T\}$. Then we can calculate separately 
	\begin{eqnarray*}
	     &&\left|\sum_{[t_i,t_{i+1}]\in C^n_\epsilon}\frac{1}{\Gamma(m)}\int_{S(t\wedge t_i)}^{S(t\wedge t_{i+1})}(f^{(m)}(r) - f^{(m)}(S(t\wedge t_i)))(S(t\wedge t_{i+1}) - r)^{m-1}dr\right| \\
	     &\leq& \sum_{[t_i,t_{i+1}]\in C^n_\epsilon}M|S(t\wedge t_{i+1}) - S(t\wedge t_i)|^p \leq \sum_{[t_i,t_{i+1}]\in \pi_n} Mg_\epsilon(S(t\wedge t_i))|S(t\wedge t_{i+1}) - S(t\wedge t_i)|^p	
	\end{eqnarray*}
	Here $g_\epsilon$ is a continuous taking value 1 on $[k-\epsilon,k+\epsilon]$ for each $k \in \Gamma_f^c$ and value 0 outside of $\cup_{k\in\Gamma_f^c}[k-2\epsilon,k+2\epsilon]$. Here we also take $\epsilon$ small enough such that $[k-2\epsilon,k+2\epsilon]$ are disjoint and $||g_\epsilon||_\infty \leq 1$. Then by definition of the $p-$th order variation, we see 
	\begin{eqnarray*}
	    &&\limsup_{n\rightarrow\infty} \left|\sum_{[t_i,t_{i+1}]\in C^n_\epsilon}\frac{1}{\Gamma(m)}\int_{S(t\wedge t_i)}^{S(t\wedge t_{i+1})}(f^{(m)}(r) - f^{(m)}(S(t\wedge t_i)))(S(t\wedge t_{i+1}) - r)^{m-1}dr\right| \\
	    &\leq& \int_0^tMg_\epsilon(S(r))d[S]^p(r)
	\end{eqnarray*}
	now let $\epsilon\rightarrow0$, we obtain that 
	\begin{eqnarray*}
	     &&\lim_{\epsilon\rightarrow0}\limsup_{n\rightarrow\infty} \left|\sum_{[t_i,t_{i+1}]\in C^n_\epsilon}\frac{1}{\Gamma(m)}\int_{S(t\wedge t_i)}^{S(t\wedge t_{i+1})}(f^{(m)}(r) - f^{(m)}(S(t\wedge t_i)))(S(t\wedge t_{i+1}) - r)^{m-1}dr\right| \\
	     &\leq& \sum_{k\in\Gamma_f^c} \int_0^tM\1_{\{S(r) = k\}}d[S]^p(r)
	\end{eqnarray*}
	By assumption $\ref{ass1}$ we then see that 
	\[
	    \lim_{\epsilon\rightarrow0}\limsup_{n\rightarrow\infty} \left|\sum_{[t_i,t_{i+1}]\in C^n_\epsilon}\frac{1}{\Gamma(m)}\int_{S(t\wedge t_i)}^{S(t\wedge t_{i+1})}(f^{(m)}(r) - f^{(m)}(S(t\wedge t_i)))(S(t\wedge t_{i+1}) - r)^{m-1}dr\right| = 0
	\]
	Next for the set $C^n_{1,\epsilon}$ and $\alpha = p - m$, we have 
    \begin{eqnarray*}
	    &&\sum_{[t_i,t_{i+1}]\in C^n_{1,\epsilon}} \frac{1}{\Gamma(m)}\int_{S(t\wedge t_i)}^{S(t\wedge t_{i+1})} (S(t\wedge t_{i+1}) - r)^{m-1}\frac{1}{\Gamma(\alpha)}\int_{S_{t\wedge t_i}}^r\frac{C_{S(t\wedge t_i)^+}^p f(s)}{(r-s)^{1-\alpha}} dsdr \\
	    &=& \sum_{[t_i,t_{i+1}]\in C^n_{1,\epsilon}}\frac{1}{\Gamma(m)\Gamma(\alpha)}\int_{S(t\wedge t_i)}^{S(t\wedge t_{i+1})}\int_{s}^{S_{t\wedge t_{i+1}}}(S(t\wedge t_{i+1}) - r)^{m-1}\frac{C_{S(t\wedge t_i)^+}^p f(s)}{(r-s)^{1-\alpha}}drds \\
	    &=& \sum_{[t_i,t_{i+1}]\in C^n_{1,\epsilon}}\frac{1}{\Gamma(p)}\int_{S(t\wedge t_i)}^{S(t\wedge t_{i+1})}(S(t\wedge t_{i+1}) - s)^{p-1}C_{S(t\wedge t_i)^+}^p f(s)ds
	\end{eqnarray*}
	Then by inequality ($\ref{continuity}$) and the fact that $f^{(p+)}(S(t_i)) = C_{S(t_i)^+}^p f(S(t_i)) = 0$(this is because $S(t_i) \in \Gamma_f$ and the lemma \ref{lemGamma_f}), we have 
	\begin{eqnarray*}
	    &&\left|\int_{S(t\wedge t_i)}^{S(t\wedge t_{i+1})}(S(t\wedge t_{i+1}) - s)^{p-1}C_{S(t\wedge t_i)^+}^p f(s)ds\right| \\
	    &\leq& \int_{S(t\wedge t_i)}^{S(t\wedge t_{i+1})}(S(t\wedge t_{i+1}) - s)^{p-1}\omega_\epsilon(s-S(t\wedge t_i))ds \\
	    &\leq& \frac{1}{p}\omega_\epsilon(S(t_{i+1}) - S(t_i))|S(t_{i+1}) - S(t_i)|^p
	\end{eqnarray*}
	Thus we obtain
	\begin{eqnarray*}
	    &&\left|\sum_{[t_i,t_{i+1}]\in C^n_{1,\epsilon}} \frac{1}{\Gamma(m)}\int_{S(t\wedge t_i)}^{S(t\wedge t_{i+1})} \frac{(S(t\wedge t_{i+1}) - r)^{m-1}}{\Gamma(\alpha)}\int_{S_{t\wedge t_i}}^r\frac{C_{S(t\wedge t_i)^+}^p f(s)}{(r-s)^{1-\alpha}} dsdr\right| \\
	    &\leq& \sum_{[t_i,t_{i+1}]\in \pi_n}\frac{\omega_\epsilon(S(t_{i+1}) - S(t_i))}{\Gamma(p+1)}|S(t_{i+1}) - S(t_i)|^p
	\end{eqnarray*}
	Since $osc(S,\pi_n)\rightarrow 0$ as $n\rightarrow\infty$, we have 
	\[
	    \lim_{n\rightarrow\infty} \sum_{[t_i,t_{i+1}]\in C^n_{1,\epsilon}} \frac{1}{\Gamma(m)}\int_{S(t\wedge t_i)}^{S(t\wedge t_{i+1})} (S(t\wedge t_{i+1}) - r)^{m-1}\frac{1}{\Gamma(\alpha)}\int_{S(t\wedge t_i)}^r\frac{C_{S(t\wedge t_i)^+}^p f(s)}{(r-s)^{1-\alpha}} dsdr = 0
	\]
	Now for the set $C_{2,\epsilon}^n$, we have
	\begin{eqnarray*}
	    &&\sum_{[t_i,t_{i+1}]\in C^n_{2,\epsilon}}\frac{1}{\Gamma(m)}\int_{S(t\wedge t_{i+1})}^{S(t\wedge t_i)}(f^{(m)}(S(t\wedge t_i)) - f^{(m)}(r))(S(t\wedge t_{i+1}) - r)^{m-1} dr \\
	    &=& \sum_{[t_i,t_{i+1}]\in C_{2,\epsilon}}\frac{1}{\Gamma(m)}\int_{SS(t\wedge t_{i+1})}^{S(t\wedge t_i)}\frac{1}{\Gamma(\alpha)}\int_r^{S(t\wedge t_i)}\frac{C_{r^+}^p f(s)}{(S(t\wedge t_i) - s)^{1-\alpha}} ds(S(t\wedge t_{i+1}) - r)^{n-1} dr \\
	    &=& \sum_{[t_i,t_{i+1}]\in C_{2,\epsilon}}\frac{1}{\Gamma(m)\Gamma(\alpha)}\int_{S(t\wedge t_{i+1})}^{S(t\wedge t_i)}\int_{S(t\wedge t_{i+1})}^s(S(t\wedge t_{i+1}) -r)^{m-1}\frac{C_{r^+}^p f(s)}{(S(t\wedge t_i) - s)^{1-\alpha}}drds
	\end{eqnarray*}
	Again by the inequality ($\ref{continuity}$), we will have 
	\[
	    |C_{r+}^p f(s)| \leq \omega_\epsilon(\sqrt{(r - S(t\wedge t_{i+1}))^2 + (s - S(t\wedge t_{i+1}))^2}) \leq \omega_\epsilon(\sqrt{2}|S(t\wedge t_{i+1}) - S(t\wedge t_i)|)
	\]
	Hence, we can obtain the estimation:
	\begin{eqnarray*}
	    &&\left|\int_{S(t\wedge t_{i+1})}^{S_{t\wedge t_i}}\int_{S(t\wedge t_{i+1})}^s(S(t\wedge t_{i+1}) -r)^{m-1}\frac{C_{r^+}^p f(s)}{(S(t\wedge t_i) - s)^{1-\alpha}}drds\right|\\
	    &\leq& \int_{S(t\wedge t_{i+1})}^{S_{t\wedge t_i}}\int_{S(t\wedge t_{i+1})}^s(S(t\wedge t_{i+1}) -r)^{m-1}\frac{\omega_\epsilon(\sqrt{2}|S(t\wedge t_{i+1}) - S(t\wedge t_i)|)}{(S(t\wedge t_i) - s)^{1-\alpha}}drds \\
	    &\leq& \frac{1}{m\alpha}\omega_\epsilon(\sqrt{2}|S(t\wedge t_{i+1}) - S(t\wedge t_i)|)|S(t_{i+1}) - S(t_i)|^p
	\end{eqnarray*}
	Thus we could obtain that 
	\[
	    \lim_{n\rightarrow\infty}\sum_{[t_i,t_{i+1}]\in C^n_{2,\epsilon}}\frac{1}{\Gamma(m)}\int_{S(t\wedge t_{i+1})}^{S(t\wedge t_i)}(f^{(m)}(r) - f^{(m)}(S(t\wedge t_i)) )(S(t\wedge t_{i+1}) - r)^{m-1} dr = 0
	\]
	And finally for the case $t_i = 0$ or $t_{i+1} = T$, the continuity of $f^{(m)}$ and $S$ will show that 
	\[
	    \lim_{n\rightarrow\infty} \int_{S(t\wedge t_i)}^{S(t\wedge t_{i+1})}(f^{(m)}(r) - f^{(m)}(S(t\wedge t_i)))(S(t\wedge t_{i+1}) - r)^{m-1}dr = 0
	\]
	Thus, by summing all these together and let $\epsilon$ tends to 0, we have 
	\[
	    \lim_{\epsilon\rightarrow 0}\lim_{n\rightarrow\infty}\sum_{[t_i,t_{i+1}]\in \pi_n}\frac{1}{\Gamma(m)}\int_{S(t\wedge t_{i+1})}^{S(t\wedge t_i)}(f^{(m)}(S(t\wedge t_i)) - f^{(m)}(r))(S(t\wedge t_{i+1}) - r)^{m-1} dr = 0
	\]
	Hence we see $\lim_{n\rightarrow\infty} L^n$ exists and we denote it as $\int_0^t (f'\circ S).dS$. And this leads to the Ito formula 
	\[
	    f(S(t)) - f(S(0)) = \int_0^t (f'\circ S).dS
	\]
    \end{proof}
    \begin{remark}\label{formito}{\em 
        We can not expect the It\^o formula in the fractional case to have the same form as in the integer case, i.e. there might be no It\^o term in the fractional case even if  Assumption $\ref{ass1}$ does not hold. For the example in  Remark \ref{counterass1}, we can show that for  $2 < p <3$ and $f(x) = |x|^p$, the It\^o term still vanishes even though we have 
        \[
            \int_{0}^1 f^{(p+)}(S(t)) d[S]^p(t) = \frac{1}{\Gamma(p+1)}\int_0^1 \1_{\{S(t) = 0\}} dc(t) = \frac{c(1) - c(0)}{\Gamma(p+1)} = \frac{1}{\Gamma(p+1)},
        \]
        since $\{S(t) = 0\} = C$ is the support of the function $c$. In this case, we will show that 
        \[
            0 = f(S(1)) - f(S(0)) = \int_0^1 (f'\circ S).dS.
        \]
         In fact, we will need to calculate 
        \[
            \lim_{n\rightarrow\infty} \sum_{[t_i,t_{i+1}]\in\pi_n}f'(S(t_i))(S(t_{i+1}) - S(t_i)) + \frac{1}{2}f''(S(t_i))(S(t_{i+1}) - S(t_i))^2
        \]
        which is in fact the 'rough integral' of $f'$ along the reduced order-p Ito rough path associated with $S$ \cite{cont2018pathwise}.
        Splitting the terms across each $I_j^i$ we have
        \begin{eqnarray*}
            &&\sum_{k=0}^{k_n - 1} p(S(t_{j,n}^{i,k}))^{p-1}\frac{1}{k_n}2^{-\frac{i}{p}} - \sum_{k=k_n}^{2k_n - 1} p(S(t_{j,n}^{i,k}))^{p-1}\frac{1}{k_n}2^{-\frac{i}{p}} \\
            &=&  -p(S(t_{j,n}^{i,k_n}))^{p-1}\frac{1}{k_n}2^{-\frac{i}{p}} = -p\cdot (2^{-\frac{i}{p}})^{p-1} \cdot \frac{2^{-\frac{i}{p}}}{k_n}
        \end{eqnarray*}
        Thus 
        \[
            \sum_{[t_i,t_{i+1}]\in\pi_n}f'(S(t_i))(S(t_{i+1}) - S(t_i)) = \sum_{i=1}^n\sum_{j=1}^{2^{i-1}} -p\cdot \frac{2^{-i}}{k_n} = -\frac{np}{2k_n}
        \]
        Now we consider the term involving the second derivative. On $I_j^i$, we have 
        \begin{eqnarray*}
            &&\sum_{[t_{j,n}^{i,k},t_{j,n}^{i,k+1}]\in\pi^i_{j,n}} \frac{1}{2}f''(S(t_{j,n}^{i,k}))(S(t_{j,n}^{i,k+1}) - S(t_{j,n}^{i,k}))^2 \\
            &=& \sum_{[t_{j,n}^{i,k},t_{j,n}^{i,k+1}]\in\pi_{j,n}^i} \frac{1}{2}p(p-1)(S(t_{j,n}^{i,k}))^{p-2} \frac{2^{-\frac{2i}{p}}}{k_n^2} \\
            &=& \frac{1}{2}\left(\sum_{k=0}^{k_n-1} p(p-1)k^{p-2}\frac{2^{-i}}{k_n^p} + \sum_{k=1}^{k_n} p(p-1)k^{p-2}\frac{2^{-i}}{k_n^p}\right)
        \end{eqnarray*}
        Summing the second order terms over $\pi_n$ we obtain: 
        \[
            \frac{np(p-1)}{4} \frac{2\sum_{k=1}^{k_n} k^{p-2} - k_n^{p-2}}{k_n^p}
        \]
        Adding together the first order term $-\frac{np}{2k_n}$, we need to calculate
        \[
            \lim_{n\rightarrow\infty}\frac{np}{4} \left((p-1)\frac{2\sum_{k=1}^{k_n} k^{p-2} - k_n^{p-2}}{k_n^p} - \frac{2}{k_n}\right)
        \]
        Since it is obvious to see that 
        \[
            \lim_{n\rightarrow\infty} \frac{np}{4}(p-1)\frac{1}{k_n^2} = \lim_{n\rightarrow\infty} \frac{np}{4}(p-1)\frac{1}{n^{\frac{2}{p-1}}}
        \]
        and $\frac{2}{p-1} > 1$ means that 
        \[
            \lim_{n\rightarrow\infty} \frac{np}{4}(p-1)\frac{1}{k_n^2} = 0
        \]
        Hence it remains to consider 
        \[
            \lim_{n\rightarrow\infty}\frac{np}{2} \left((p-1)\frac{\sum_{k=1}^{k_n} k^{p-2}}{k_n^p} - \frac{1}{k_n}\right)        
        \]
        Use the inequality
        \[
            \frac{k_n^{p-1}}{p-1} = \int_{0}^{k_n} x^{p-2}dx \leq \sum_{k=1}^{k_n} k^{p-2} \leq \int_1^{k_n+1} x^{p-2}dx = \frac{(k_n+1)^{p-1} - 1}{p-1}
        \]
        We then obtain first that 
        \[
            \lim_{n\rightarrow\infty}\frac{np}{2} \left((p-1)\frac{\sum_{k=1}^{k_n} k^{p-2}}{k_n^p} - \frac{1}{k_n}\right) \geq  \lim_{n\rightarrow\infty}\frac{np}{2} \left((p-1)\frac{k_n^{p-1}}{(p-1)k_n^{p}} - \frac{1}{k_n}\right) = 0
        \]
        and secondly 
        \begin{equation}\label{ineqexample}
            \lim_{n\rightarrow\infty}\frac{np}{2} \left((p-1)\frac{\sum_{k=1}^{k_n} k^{p-2}}{k_n^p} - \frac{1}{k_n}\right)  \leq   \lim_{n\rightarrow\infty}\frac{np}{2} \left(\frac{(k_n+1)^{p-1} - 1 - k_n^{p-1}}{k_n^p}\right)        
        \end{equation}
        Notice that 
        \[
            \lim_{n\rightarrow\infty} nk_n^{1-p} =1
        \]
        and $k_n\rightarrow\infty$, we see the right hand side of the inequality $(\ref{ineqexample})$ equals to 
        \[
            \lim_{k_n\rightarrow\infty}p\frac{k_n^{p-1}}{2} \left(\frac{(k_n+1)^{p-1} - 1 - k_n^{p-1}}{k_n^p}\right) = \frac{p}{2}\lim_{x\rightarrow\infty} \frac{(x+1)^{p-1} - 1 - x^{p-1}}{x}
        \]
        By l'Hopital's rule, above limit equals to 
        \begin{eqnarray*}
            &&\lim_{x\rightarrow\infty} \frac{p(p-1)}{2}((x+1)^{p-2} - x^{p-2}) \\
            &=&\frac{p(p-1)}{2} \lim_{x\rightarrow\infty} x^{p-2} \left(\left(1+\frac{1}{x}\right)^{p-2} - 1\right) \\
            &=& \frac{p(p-1)}{2} \lim_{x\rightarrow\infty} x^{p-2} \frac{p-2}{x} = 0
        \end{eqnarray*}
        since $p-2 < 1$. Thus in this case we see the remainder term is  zero, which is different from the integer case. Hence we see that \textbf{the pathwise change of variable formula  in Theorem $\ref{itoformula1}$  may hold even if  Assumption $\ref{ass1}$ does not hold.}
        }
    \end{remark}
    We now give an example of  path with the same $p-$th variation as above but leading to a non-zero remainder term in the change of variable formula.
    \begin{example}\label{nonzero}
        Define the intervals $I^i_j$ as in \eqref{eq.Iji}. Then we define $S|_{I^i_j}$ by induction on $i$. 
        Let $g(t) = 2min\{t,1-t\}$ on $[0,1]$ and $0$ otherwise. For  $a <b$, we define   $g(t,[a,b]) = g\left(\frac{t-a}{b-a}\right)$.
        
        First for $i = 1$, we divide the interval $I^i_j$ into $r_i$ smaller intervals $I^i_{j,k}, k = 1,\cdots,r_i$ such that $|I^i_{j,k}| = \frac{|I^i_{j}|}{r_i}$ and two of each are non-intersecting. On each interval $I^i_{j,k}$, we define $S(t) = 2^{-i}g(t,I^i_{j,k})$ for $k = 1,\cdots, r_i, j = 1,\cdots, 2^{i-1}$. In other words,   $S$ is defined as the limit
        \[
            S(t) = \sum_{i=1}^\infty\sum_{j=1}^{2^{i-1}} \sum_{k=1}^{r_i} 2^{-i}g(t,I^i_{j,k})
        \]
        Let $\pi_n=(\tau^n_l)$ be the dyadic Lebesgue partition associated with $S$:
        $$ \tau^n_0=0,\qquad \tau^n_{l+1}=\inf\{ t>\tau^n_l,\ |S(t)-S(\tau^n_l)|> 2^{-n}\}$$
        \[
            [S]^p(1) =\lim_{n\rightarrow\infty} \sum_{\pi_n}|S(\tau^n_{l+1})-S(\tau^n_l)|^p = \lim_{n\rightarrow\infty} \sum_{i=1}^n\sum_{j=1}^{2^{i-1}} \sum_{k=1}^{r_i} 2\times2^{-np}\times 2^{n-i} = \lim_{n\rightarrow\infty} \sum_{i=1}^n r_i 2^{-np+n}
        \]
      
        We can choose 
        \[
            r_i = \lfloor 2^{(i-1)(p-1)}(2^{p-1} - 1)\rfloor
        \]
        so that  
           \[
            1 = \lim_{n\rightarrow\infty} \sum_{i=1}^n r_i 2^{-np+n}.
        \]
        
        Then similar to what was discussed in the above remark, we have $[S]^p(t) = c(t)$, the Cantor function. Let $f(x) = |x|^p, 2 < p < \frac{3}{2}$ and $T = 1$, we are going to calculate the It\^o remainder term for $f(S(1)) - f(S(0))$. We calculate the limit 
        \begin{eqnarray*}
            &&\lim_{n\rightarrow\infty} \sum_{[t_i,t_{i+1}] \in \pi_n} p(S(t_i))^{p-1}(S(t_{i+1}) - S(t_i)) + \frac{p(p-1)}{2}(S(t_i))^{p-2}(S(t_{i+1}) - S(t_i))^2 \\
            &=& \lim_{n\rightarrow\infty} \sum_{i=1}^n\sum_{j=1}^{2^{i-1}} r_i\left(-p2^{-n}(2^{-i})^{p-1} + \frac{p(p-1)}{2}2^{-2n}c_n^i\right)
        \end{eqnarray*}
        Here 
        \[
            c_n^i = 2\sum_{k=1}^{2^{n-i}-1} 2^{-n(p-2)}k^{p-2} + 2^{-n(p-2)}(2^{n-i})^{p-2}
        \]
        We have 
        \[
            c_n^i < 2^{1-n(p-2)}\int_0^{2^{n-i}}x^{p-2}dx +2^{-n(p-2)}(2^{n-i})^{p-2}
        \]
        and hence 
        \[
            \frac{p(p-1)}{2}2^{-2n}c_n^i < p2^{-n-ip+i} +\frac{p(p-1)}{2}2^{-np + (n-i)(p-2)}
        \]
        Hence 
        \begin{eqnarray*}
            &&\sum_{i=1}^n\sum_{j=1}^{2^{i-1}} r_i\left(-p2^{-n}(2^{-i})^{p-1} + \frac{p(p-1)}{2}2^{-2n}c_n^i\right) \\
            &<& \sum_{i=1}^{n-1}2^{i-1} r_i\left(\frac{p(p-1)}{2}2^{-2n-ip+2i}\right) + 2^{n-1}r_n\left(-p2^{-np} + \frac{p(p-1)}{2}2^{-np}\right)\\
            &<& \sum_{i=1}^{n-1}2^{-p}(2^{p-1} - 1)\left(\frac{p(p-1)}{2}2^{-2n+2i}\right) + 2^{-p}(2^{p-1} - 1)\left(-p + \frac{p(p-1)}{2}\right)  \\
            &=&p2^{-p}(2^{p-1} - 1)\left(\frac{p-1}{2}\frac{1 - (\frac{1}{4})^{n-1}}{3} + -1 + \frac{p-1}{2}\right) 
        \end{eqnarray*}
        Now we let $n\rightarrow\infty$, we have 
        \[
            \lim_{n\rightarrow\infty} \sum_{i=1}^n\sum_{j=1}^{2^{i-1}} r_i\left(-p2^{-n}(2^{-i})^{p-1} + \frac{p(p-1)}{2}2^{-2n}c_n^i\right) \leq p2^{-p}(2^{p-1} - 1)\left(\frac{2(p-1)}{3} - 1\right)
        \]
        So if $2 < p < \frac{5}{2}$, then 
        \[
            \int_0^1 f'(S(t))dS(t) < 0
        \]
        but $f(S(1)) - f(S(0)) = 0$, which means there will be a non-zero remainder. And in this case, we also have 
        \[
            \int_0^1 \1_{\{S(t) = 0\}}d[S]^p(t) = 1
        \]
        so Assumption $\ref{ass1}$ is not satisfied. 
    \end{example}
    In fact, we can provide a formula for the It\^o remainder term for this path and function $f = |x|^p$. Take $T = 1$, $m = \lfloor p\rfloor$ and $\alpha = p - m$. Let 
    \[
        G_f^p(a,b) = \frac{1}{(m-1)!|b-a|^p}\int_{a}^b (f^{(m)}(x) - f^{(m)}(a))(b-x)^{m-1}dx
    \]
    for $a \neq b$ and take the limit value when exists for $a = b$. For the function $f(x) = |x|^p$, we can see $G_f^p$ is defined on $\mathbb{R}^2/\{(0,0)\}$ and for $k > 0$, 
    \[
        G_f^p(ka,kb) = G_f^p(a,b)
    \]
    which indicate us to consider $G_f^p$ as a function on the unit circle $S^1$. We define the projection map $P: \mathbb{R}^2/\{(0,0)\}\rightarrow S^1$ by $p(x) = \frac{x}{||x||}$ and define a sequences of measures by 
    \[
        \tilde{\nu}_n:= \frac{1}{N_n}\sum_{i=0}^{N_n - 1} \delta_{P(S(t_i),S(t_{i+1}))}
    \]
    where $N_n$ is the number of intervals in the partition $\pi_n$. Furthermore, we define 
    \[
        \hat{G}_f^p(\theta) = G_f^p(\cos(\theta),\sin(\theta))
    \]
    for $\theta \in [0,2\pi)$. Since the distance between two successive points in partition $\pi_n$ is $2^{-n}$, the remainder term can be rewritten as
    \[
        N_n\cdot2^{-np}\int \hat{G}_f^p(x)\tilde{\nu}_n(dx)
    \]
    and we have $N_n = \sum_{i=1}^n\sum_{j=1}^{2^{i-1}}\sum_{k=1}^{r_i}2^{n-i+1} = 2^n\sum_{i=1}^n r_i$. Hence the remainder term is 
    \[
        2^{n-np}\sum_{i=1}^n r_i\int \hat{G}_f^p(x)\tilde{\nu}_n(dx)
    \]
    By definition, we have $\lim_{n\rightarrow\infty}2^{n-np}\sum_{i=1}^n r_i = 1$. And if we have the weak convergence of $\tilde{\nu}_n$, then we can obtain the limit expression. Now let's calculate the limit of the measures $\tilde{\nu}_n$. 
    
    Since $S(t_i) = k2^{-n}$ for some positive number $k$ and $t_i \in \pi_n$, the result $\tan\left(P(S(t_i),S(t_{i+1}))\right)$ will equal to either $\frac{k}{k+1}$ or $\frac{k}{k-1}$, which means the final measure $\tilde{\nu}$ will be supported on these points. And 
    \begin{eqnarray*}
        \tilde{\nu}\left(\arctan\left(\frac{k}{k+1}\right)\right) &=& \lim_{n\rightarrow\infty} \tilde{\nu}_n\left(\arctan\left(\frac{k}{k+1}\right)\right) \\
        &=&  \frac{1}{2^{\lceil \log_2(k+1)\rceil p}}\frac{2^{p-1}-1}{2^p-1}
    \end{eqnarray*}
    By symmetry, we have 
    \[
        \tilde{\nu}\left(\arctan\left(\frac{k+1}{k}\right)\right) = \tilde{\nu}\left(\arctan\left(\frac{k}{k+1}\right)\right) = \frac{1}{2^{\lceil \log_2(k+1)\rceil p}}\frac{2^{p-1}-1}{2^p-1}
    \]
    Hence we have the It\^o formula for the path $S$ defined in the example $\ref{nonzero}$ and $f(x) = |x|^p$:
    \begin{equation}\label{nonzeroFormula}
        f(S(1)) - f(S(0)) = \int_{0}^1 f'(S(u))dS(u) + \int_{S^1} \hat{G}_f^p(x)\tilde{\nu}(dx)
    \end{equation}
    
    \textbf{The above calculations show that the remainder term may be non-standard when assumption $\ref{ass1}$ is not satisfied.}

    Furthermore, through the calculation of exact expression of It\^o's remainder term for the path in example $\ref{nonzero}$, we can give a general method to calculate It\^o's remainder term for paths satisfying some assumptions. 

    For a function $f$ and a positive number $p = m + \alpha$, where $m = \lfloor p \rfloor$, we define the map
    \begin{eqnarray}
          G_f^p(a,b) = \frac{1}{(m-1)!|b-a|^p}\int_{a}^b (f^{(m)}(x) - f^{(m)}(a))(b-x)^{m-1}dx \label{eq.Gf}
    \end{eqnarray}
    Then $G_f^p$ defines a map from $\{(x,y)\in\mathbb{R}^2,\ x\neq y\}$. If $f^{(m)}$ is continuous, it is easy to see $G_f^p$ is a continuous function on $\{(x,y)\in\mathbb{R}^2,\ x\neq y\}$. 
    
    Denote by $N_n^t$ the number of intervals in $\pi_n$ on $[0,t]$.
    
    In order to compute the remainder term in the fractional Ito formula, we 'stratify' the increments across the partition by the values of $G_f^p$, to build an auxiliary quotient space $X$ which is required to have certain properties:
    \begin{assumption}[Auxiliary quotient space]\label{QuotientSpaceDef}
        There exists a space $X$ and a map $$ P^p_f: \{(x,y)\in\mathbb{R}^2,\ x\neq y\} \rightarrow X$$ such that 
        \begin{enumerate}
            \item[(i)] For all $x\in X$, the map $G^p_f$ defined by \eqref{eq.Gf} is constant on $(P^p_f)^{-1}(x)$;
            \item[(ii)] The sequence of measures on $[0,T]\times X$ defined by 
            \[
                \nu_n(dt,dx) :=\sum_{[t_i,t_{i+1}]\in\pi_n} |S(t_{i+1}) - S(t_i)|^p \delta_{t_i}(t)\delta_{P^p_f(S(t_i),S(t_{i+1}))}(x)
            \]
        converges weakly to a measure $\nu$ on $[0,T]\times X$. 
        \end{enumerate}   
        
    \end{assumption}
    \begin{remark}
        The measure $\tilde{\nu}$ in the example $\ref{nonzero}$ is actually the measure $\nu([0,1],dx)$ defined in assumption \ref{QuotientSpaceDef} on the space $X$. In fact, due to the fact that partitions in the example $\ref{nonzero}$ are of Lebesgue type, we have 
        \[
            \tilde{\nu}_n(dx) = \frac{1}{[S]^{p}_n(1)}\nu_n([0,1],dx)
        \]
        Here $[S]^{p}_n(1) = \sum_{[t_i,t_{i+1}]\in\pi_n, t_i \leq 1} |S(t_{i+1}) - S(t_i)|^p$. Let $n\rightarrow\infty$, we then obtain
        \[
            \tilde{\nu}(dx) = \frac{1}{[S]^p(1)}\nu([0,1],dx) = \nu([0,1],dx)
        \]
        Hence assumption \ref{QuotientSpaceDef} is a reasonable condition in order to calculate Ito's remainder term such as example $\ref{nonzero}$.
        
        Furthermore, since the first marginal measure of $\nu_n$ is
        \[
            \nu_n(dt,X) = \sum_{[t_i,t_{i+1}]\in\pi_n} |S(t_{i+1}) - S(t_i)|^p \delta_{t_i}(t)
        \]
        we see $S \in V^p(\pi_n)$ is equivalent to the weak convergence of $\nu_n(dt,X)$ on the interval $[0,T]$. In this sense, we see assumption \ref{QuotientSpaceDef} actually requires other properties in addition to the existence of $p-$th variation along the sequence $\{\pi_n\}$.
    \end{remark}
    
    The choice of  $X$ is crucial in making the calculation of the remainder term tractable.
       One can always choose $X=\mathbb{R}$ but this choice may not lead to easy calculations (see below).  
    \begin{theorem}\label{EquiDistantFormula}
        Let $\{\pi_n\}$ be a sequence of partitions and $S \in V_p(\pi)$ such that there exists $(X,P^p_f)$ satisfying Assumption \ref{QuotientSpaceDef}. We define the map $\hat{G}_f^p: X\rightarrow\mathbb{R}$ such that $\hat{G}_f^p\circ P_f^p = G_f^p$. Then
        \[
            f(S(t)) - f(S(0)) = \int_0^t f'(S(u))dS(u) + \int_0^t\int_X \hat{G}_f^p(x)\nu(ds,dx)
        \]
        where 
    \end{theorem}
    \begin{proof}
        Using a  Taylor expansion, we have actually
        \begin{eqnarray*}
            f(S(t)) - f(S(0)) &=& \sum_{[t_i,t_{i+1}]\in\pi_n,t_i\leq t} \sum_{j=1}^m \frac{f^{(j)}(S(t\wedge t_i))}{j!}(S(t\wedge t_{i+1}) - S(t\wedge t_i))^j \\
            &+& \sum_{[t_i,t_{i+1}]\in\pi_n,t_i\leq t} G_f^p(S(t_i),S(t_{i+1}))|S(t_{i+1})-S(t_i)|^p
        \end{eqnarray*}
        by the definition of measures $\nu_n$, we have 
        \begin{eqnarray*}
            f(S(t)) - f(S(0)) &=& \sum_{[t_i,t_{i+1}]\in\pi_n,t_i\leq t} \sum_{j=1}^m \frac{f^{(j)}(S(t\wedge t_i))}{j!}(S(t\wedge t_{i+1}) - S(t\wedge t_i))^j \\
            &+&  \int_0^t\int_X \hat{G}_f^p(x)\tilde{\nu}_n(dsdx)
        \end{eqnarray*}
       Using the weak convergence of $\nu_n$ in the assumption $\ref{QuotientSpaceDef}$, so taking limits yields the desired result.
    \end{proof}
   
    \begin{remark}{\em
        We can perturb the function $f$ without changing the choice of the space $X$, by working in a neighborhood of the singularity of $G_f^p$.
      Take $S^1$ as an example. Suppose $f\sim |x|^p$ near zero, which means that $G_f^p$ has a unique singularity at $(0,0)$. If furthermore, on each ray emanating from the origin, the function $G_f^p$ converges as point tends to origin 'uniformly'. Here uniform means that for each $\epsilon > 0$, there exists $\delta >0$ such that on the band $\{(a,b):d((a,b),y=x) < \delta\}$, we have $|G_f^p(a,b) - G_f^p(c,d)|\leq \epsilon$ if $P_f^p(a,b) = P_f^p(c,d)$ and $(a,b),(c,d)$ are on the band. Above assumption describe the behaviour of $G_f^p$ around the singular point $(0,0)$ and can be used to calculate the It\^o remainder term. For such function $f$ and path $S$, we can derive the formula
        \[
            f(S(t)) - f(S(0)) = \int_0^t f'(S(u))dS(u) +  \int_{S^1} g(x)\nu([0,t],dx)
        \]
        Here $g(x)$ is defined as the radial limit at the origin of $G_f^p$.
        
        Furthermore if $f$ admits a bounded $p-th$ order Caputo derivative, then the uniform convergence can be checked by uniform convergence of the Caputo derivative $C_a^p f(b)$. }
    \end{remark}
    
    We can see from the theorem that when It\^o's remainder term is non-zero, the formula is not in the classical form. Hence we will need more information to describe the It\^o remainder term in the fractional case. 
    
    Let us go back to zero remainder term case. Consider the set $(E^p)'$ of all functions satisfying assumptions \ref{ass3} and \ref{ass4}. Then $(E^p)'$ is a subset of $C^p(\mathbb{R})$. Let $E^p$ be the closure of $(E^p)'$ in $C^p(\mathbb{R})$ equipped with the semi-norms
    \[
        \sup_{x\in K}\{f^{(k)}(x):k=0,\cdots,m\} + \sup_{x,y\in K}\left\{\frac{|f^{(m)}(y) - f^{(m)}(x)|}{|y-x|^\alpha}\right\}
    \]
    here $p = m + \alpha$ with $m < p < m+1$. Then we have the following theorem
    \begin{theorem}\label{itoformula2}
         Let $S\in V_p(\pi)$  satisfying Assumption $\ref{ass1}$ and $f\in E^p$. Then   for any $t \in [0,T]$
        \[
	        f(S(t)) = f(S(0)) + \int f'(S(u))dS(u)
	    \]
	    where the integral is a limit of compensated Riemann sums of order $m = \lfloor p \rfloor$:
	    \[
	        \int f'(S(u))dS(u) = \lim_{n\rightarrow\infty}\sum_{t_i\in\pi_n} \sum_{j=1}^m \frac{f^{(j)}(S(t_i))}{j!}(S(t_{j+1}) - S(t_j))^j.
	    \]
	 
    \end{theorem}
    \begin{proof}
        Suppose $f_k\in (E^p)'$ and $f_k\rightarrow f$ in $C^p$. We have 
        \[
            f(S(t)) - f(S(0)) = L^n + \sum_{t_i\in\pi_n} \frac{1}{\Gamma(m)}\int_{S(t_i)}^{S(t_{i+1})}(f^{(m)}(r) - f^{(m)}(S(t_i)))(S(t_{i+1}) - r)^{m-1}dr
        \]
        Let $K$ be a compact set  containing the path of $S([0,T])$. Since $f_k\rightarrow f$ in $C^p$, we have for every $\epsilon > 0$, there exists $N > 0$ and for all $k > N$, 
        \[
            \sup_{x,y\in K}\frac{|(f^{(m)} - f_k^{(m)})(y) - (f^{(m)} - f_k^{(m)})(x)|}{|y-x|^\alpha} \leq \epsilon
        \]
        Hence 
        \begin{eqnarray*}
            &&\left|\sum_{[t_i,t_{i+1}]\in\pi_n} \frac{1}{\Gamma(m)}\int_{S(t_i)}^{S(t_{i+1})}((f^{(m)}-f_k^{(m)})(r) - (f^{(m)} - f_k^{(m)})(S(t_i)))(S(t_{i+1}) - r)^{m-1}dr\right| \\
            &\leq& \epsilon\sum_{[t_i,t_{i+1}]\in\pi_n} \frac{1}{\Gamma(m)}\int_{S(t_i)}^{S(t_{i+1})}|r-S(t_i)|^\alpha|S(t_{i+1}) - r|^{m-1}dr \\
            &\leq& \epsilon\sum_{[t_i,t_{i+1}]\in\pi_n} \frac{1}{m!}|S(t_{i+1}) - S(t_i)|^p
        \end{eqnarray*}
        Hence we have 
        \begin{eqnarray*}
            &&\limsup_{n\rightarrow\infty} \left|\sum_{[t_i,t_{i+1}]\in\pi_n} \frac{1}{\Gamma(m)}\int_{S(t_i)}^{S(t_{i+1})}(f^{(m)}(r) - f^{(m)}(S(t_i)))(S(t_{i+1}) - r)^{m-1}dr \right| \\
            &\leq& \limsup_{n\rightarrow\infty} \left|\sum_{[t_i,t_{i+1}]\in\pi_n} \frac{1}{\Gamma(m)}\int_{S(t_i)}^{S(t_{i+1})}(f_k^{(m)}(r) - f_k^{(m)}(S(t_i)))(S(t_{i+1}) - r)^{m-1}dr \right| + \frac{\epsilon}{m!}[S]^p(t)
        \end{eqnarray*}
        for all $k > N$. Since $f_k \in E$, we then have 
        \[
            \limsup_{n\rightarrow\infty} \left|\sum_{[t_i,t_{i+1}]\in\pi_n} \frac{1}{\Gamma(m)}\int_{S(t_i)}^{S(t_{i+1})}(f^{(m)}(r) - f^{(m)}(S(t_i)))(S(t_{i+1}) - r)^{m-1}dr \right|\leq  \frac{\epsilon}{m!}[S]^p(t)
        \]
        Finally let $\epsilon$ tends to $0$, we have 
        \[
            \lim_{n\rightarrow\infty} \left|\sum_{[t_i,t_{i+1}]\in\pi_n} \frac{1}{\Gamma(m)}\int_{S(t_i)}^{S(t_{i+1})}(f^{(m)}(r) - f^{(m)}(S(t_i)))(S(t_{i+1}) - r)^{m-1}dr \right| = 0
        \]
        Hence $L^n$ converges when $n$ tends to infinity and we have 
        \[
            f(S(t)) - f(S(0)) = \int_0^t f'(S(u))dS(u)
        \]
    \end{proof}
    \begin{remark}
        It can be seen in the proof of the theorem \ref{itoformula1} that we only need 
        \begin{equation}\label{pathpoint}
            \int_0^T\1_{\{S(t) = k\}}d[S]^p(t) = 0
        \end{equation}
        for $k \in \Gamma_f^c$. And we can denote $P_f$ as the set of paths which satisfies equation ($\ref{pathpoint}$) on $\Gamma_f^c$.
        
        On the other hand, if we fix the path $S$, then by Fubini's theorem, we actually have that
        \[
            \int_\mathbb{R}\int_0^T\1_{\{S(t)=k\}}d[S]^p(t)dk = 0
        \]
        which means equation ($\ref{pathpoint}$) is satisfied by almost surely all $k \in \mathbb{R}$. Hence we can consider the set $(E_S^p)'$ with all functions $f$ such that $\Gamma_f$ contains all points that break the equation ($\ref{pathpoint}$). And then we can obtain the closure $E_S^p$ for fixed path $S$. We can also ignore $p$ in the subscript because $p$ is included in the information of $S$. So we can just write $E_S$. And it is obvious that $(E_S)' \subset (E^p)'$ and $E_S \subset E^p$.
        
         Furthermore, by positivity of $\int_\mathbb{R}\int_0^T\1_{\{S(t)=k\}}d[S]_t^p dk$, we can see that the set of $k$ which violates equation ($\ref{pathpoint}$) is actually countable. Hence we could define set $(E_x^p)'$ to be all functions $f$ such that $\Gamma_f$ contains point $x$ and $E_x^p$ to be the completion under H\"older norm. If we furthermore define $R_S = \left\{k\in\mathbb{R}: \int_0^T\1_{\{S(t)=k\}}d[S]^p(t) \neq 0\right\}$, then we have actually $E_S = \cap_{x \in R_S} E_x^p$.
    \end{remark}
    Now let's give some examples of functions belonging to $E^p$.
    \begin{example}[Examples of functions belonging to $E^p$.]
    \noindent
    \begin{enumerate}
        \item All functions $f \in C^{m+1}(\mathbb{R})$.
        \item The function $f(x) = |x-k|^p$ for some $k \in \mathbb{R}$. It can be seen in the example \ref{xalpha} that $C_{a^+}^p f(x)$ is continuous on $\{(a,x) \in U\times U: a \leq x\}$ for any compact set $U$ contained in positive real line or negative real line, which means $\Gamma_f = \mathbb{R}\slash\{0\}$ and hence is a function in $E_p' \subset E_p$.
        \item The linear combinations of $|x-k|^p$. For example
        \[
            f(x) = \sum_{n=1}^\infty \frac{1}{n^2}|x-x_n|^p
        \]
        with $\{x_n\}$ an ordered set of rationals in $[0,1]$.
    \end{enumerate}
    \end{example}
    \begin{remark}
        Smooth ($C^\infty$) functions belong to $E_p$. Denote by  $C^{p+}(\mathbb{R})\subset E_p$ the completion of the smooth function under H\"{o}lder norm i.e. the set of functions $f \in C^p$ such that on every compact set $K$:
          \[
            \lim_{\delta\rightarrow0} \sup_{x,y\in K} \left\{\frac{|f^{(m)}(x) - f^{(m)}(y) |}{|x-y|^\alpha}: |x-y|\leq \delta\right\} = 0
        \]
        Then we can easily see that for any $q > p$, we will have $E_q \subset C^q \subset C^{p+} \subset E_S \subset E_p$ for any path $S$. Here $q$ is a non-integer.
    \end{remark}
    
    \begin{example}[Examples for $P_f$ with $f(x) = |x|^p$]
    \noindent
    \begin{enumerate}
        \item Fractional Brownian Motion: Let $p=1/H$. We have 
        \[
            \mathbb{E}\left[\int_0^T \1_{\{B^H(t) = 0\}}dt\right] = 0
        \]
        which means that almost surely $B^H \in P^p_f$.
    \end{enumerate}
    \end{example}
    \begin{remark}
        We give here another remark on Assumption $\ref{ass4}$. In fact, as we see in the remark $\ref{rem.ass4}$, Assumption $\ref{ass4}$ means uniform convergence of the equation 
        \[
            \lim_{y\rightarrow x} \frac{f^{(m)}(y) - f^{(m)}(x)}{|y-x|^\alpha} = 0
        \]
        on every compact set $K$. Hence, according to the proof of the theorem $\ref{itoformula1}$, we can change 3 assumptions to the following one
        \begin{assumption}\label{ass5}
            There exists a sequence of open set $U_i$ such that $\bar{U}_{i+1} \subset U_{i}$ and
            \[
                \lim_{y\rightarrow x} \frac{f^{(m)}(y) - f^{(m)}(x)}{|y-x|^\alpha} = 0
            \]
            locally uniformly outside of $U_i$. Furthermore, Let $C = \cap_{i}^\infty \bar{U}_{i}$, we have 
            \[
                \int_o^T \1_{\{S(t) \in C\}} d[S](t) = 0
            \]  
        \end{assumption}
        If we replace the 3 assumptions in the theorem $\ref{itoformula1}$ by the single Assumption $\ref{ass5}$, then Theorem $\ref{itoformula1}$ still holds. And in this case we do not need to calculate the Caputo derivative of the function $f$. The only thing that we need to check is the uniform H\"older continuity. This means we have 
    \end{remark}
    \begin{theorem}
        Suppose $S\in V_p(\pi)$ and $f$ satisfies assumption \ref{ass5} for $m = \lfloor p \rfloor$. Then  
        \[
	     \forall t\in [0,T],\qquad   f(S(t)) = f(S(0)) + \int_0^T (f'\circ S).dS
	    \]
	    where the  last term is a F\"ollmer integral, defined a limit of compensated Riemann sums of order $m = \lfloor p \rfloor$:
	    \[
	       \int_0^T (f'\circ S).dS = \lim_{n\rightarrow\infty}\sum_{t_i\in\pi_n} \sum_{j=1}^m \frac{f^{(j)}(S(t\wedge t_i))}{j!}(S(t\wedge t_{j+1}) - S(t\wedge t_j))^j.
	    \]
    \end{theorem}
 
        This is a special case of Theorem \ref{ItoFunctionalFormula} on the path-dependent   It\^o formula, whose proof will be given in Sec. \ref{sec.pathdependent}.

    \subsection{Change of variable formula for time-dependent functions}\label{sec.timedependent}
    We now extend the change of variable formula to time dependent functions such as $f(t,x) = |x-g(t)|^p$ where $g$ is a smooth function.
    
    For $f \in C^{1,p}([0,T]\times\mathbb{R})$ right-differentiable in the time variable, we define the open set  
    \begin{eqnarray*}
        \Gamma_f = &&\{(t,x) \in [0,T]\times \mathbb{R}: \exists \text{ open set }U \ni (t,x) \text{ such that }C_{a+}^p f(s,b) \text{ is continuous on } \\
        &&\{(a,b) \in U_s\times U_s: a \leq b\} \text{ uniformly on }P_t(U)\}
    \end{eqnarray*}
    Here $U_s$ means the intersection of $U$ and $t=s$ and $P_t(U)$ is the projection of $U$ to $t-$axis. For the change of variable formula for time-dependent functions, we need the following assumption:
    \begin{assumption}\label{AssGammafS}
    \[
        \int_0^T \1_{\{(t,S(t)) \notin \Gamma_f\}}d[S]^p(t) = 0
    \]
    \end{assumption}
    
    \begin{theorem}\label{itoformula3}
        Let $f \in C^{1,p}([0,T]\times\mathbb{R})$ with left local derivative on $\mathbb{R}$ on $[0,T]$, $S\in V_p(\pi)$. Under  Assumption $\ref{AssGammafS}$, we have
        \[
            f(t,S(t)) - f(0,S(0)) = \int_0^t \partial_t f(u,S(u))du + \int_0^t D_xf(u,S(u))dS(u)
        \]  
        where $m = \lfloor p\rfloor$ and 
        \[
            \int_0^t D_x f(u,S(u))dS(u) = \lim_{n\rightarrow\infty}\sum_{t_i \in \pi_n}\sum_{j=1}^{m}\frac{f^{(j)}(t\wedge t_i, S(t\wedge t_i))}{j!}(S(t\wedge t_{i+1}) - S(t\wedge t_i))^j.
        \]
    \end{theorem}
    \begin{proof}
        For simplicity, we only prove the case $t=T$ here. As before we write 
        \begin{eqnarray*}
            f(T,S(T)) - f(0,S_0) &=& \sum_{[t_i,t_{i+1}]\in\pi_n} (f(t_{i+1},S(t_{i+1})) - f(t_i,S(t_i))) \\
            &=& \sum_{[t_i,t_{i+1}]\in\pi_n} (f(t_{i+1},S(t_{i+1})) - f(t_i,S(t_{i+1}))) \\
            &+& \sum_{[t_i,t_{i+1}]\in\pi_n} (f(t_{i},S(t_{i+1})) - f(t_i,S(t_i)))
        \end{eqnarray*}
        And for the first sum we have 
        \[
            \sum_{[t_i,t_{i+1}]\in\pi_n} (f(t_{i+1},S(t_{i+1})) - f(t_i,S(t_{i+1}))) = \sum_{[t_i,t_{i+1}]\in\pi_n}\int_{t_i}^{t_{i+1}}\partial_tf(u,S(t_{i+1})) du
        \]
        Let's write $S^n(t) = \sum_{t_i\in\pi_n} S(t_{i+1})\1_{[t_i,t_{i+1})}(t) + S(T)\1_{\{T\}}(t)$. Then above sum can be written as 
        \begin{eqnarray*}
            &&\sum_{[t_i,t_{i+1}]\in\pi_n} (f(t_{i+1},S(t_{i+1})) - f(t_i,S(t_{i+1}))) \\
            &=& \sum_{[t_i,t_{i+1}]\in\pi_n}\int_{t_i}^{t_{i+1}}\partial_tf(u,S^n(u)) du = \int_0^T \partial_t f(u,S^n(u))du
        \end{eqnarray*}
        Now since $||S^n - S||_{\infty} \rightarrow 0$ as $n\rightarrow\infty$ due to continuity of $S$, we see that 
        \[
            \lim_{n\rightarrow\infty}\sum_{[t_i,t_{i+1}]\in\pi_n} (f(t_{i+1},S(t_{i+1})) - f(t_i,S(t_{i+1}))) = \int_0^T \partial_t f(u,S(u))du
        \]
        Now let's deal with the second sum
        \begin{eqnarray*}
	    &&\sum_{[t_i,t_{i+1}]\in\pi_n} f(t_i,S(t_{i+1})) - f(t_i,S(t_i)) \\
	    &=&  \sum_{[t_i,t_{i+1}]\in\pi_n}\left(f(t_i,S(t_{i+1})) - f(t_i,S(t_i)) - \cdots - \frac{D_x^mf(t_i,S( t_i))}{n!}(S(t_{i+1}) - S(t_i))^n\right) \\
	    &+& \sum_{[t_i,t_{i+1}]\in\pi_n} \sum_{j=1}^m \frac{D_x^jf(t_i,S(t_i))}{j!}(S(t_{j+1}) - S(t_j))^j \\
	    &=& L^n + \sum_{[t_i,t_{i+1}]\in\pi_n} \frac{1}{\Gamma(m)}\int_{S(t_i)}^{S( t_{i+1})}(D_x^mf(t_i,r) - D_x^mf(t_i,S(t_i)))(S(t_{i+1}) - r)^{m-1}dr
	\end{eqnarray*}
	Here $L^n = \sum_{[t_i,t_{i+1}]\in\pi_n} \sum_{j=1}^m \frac{D_x^jf(t_i,S(t_i))}{j!}(S(t_{j+1}) - S(t_j))^j$. Now let's define set $C^n_\epsilon = \{[t_i,t_{i+1}]\in\pi_n: (t_i,S_{t_i}) \notin \Gamma_f^\epsilon\}$. Here $\Gamma_f^\epsilon = \{(t,x) \in \Gamma_f: d((t,x),\partial\Gamma_f) > \epsilon\}$. Then $(\Gamma_f^\epsilon)^c := [0,T]\times\mathbb{R}\slash\Gamma_f^\epsilon$ is a closed set. We have for $\alpha = p-m$, 
	\begin{eqnarray*}
	    &&\left|\sum_{[t_i,t_{i+1}]\in C_\epsilon^n} \frac{1}{\Gamma(m)}\int_{S(t_i)}^{S( t_{i+1})}(D_x^mf(t_i,r) - D_x^mf(t_i,S(t_i)))(S(t_{i+1}) - r)^{m-1}dr\right| \\
	    &\leq& \left|\sum_{[t_i,t_{i+1}]\in C_\epsilon^n} \frac{1}{\Gamma(m)}\int_{S(t_i)}^{S( t_{i+1})}M|r-S(t_i)|^\alpha|S(t_{i+1}) - r|^{m-1}dr\right| \\
	    &=& \sum_{[t_i,t_{i+1}]\in C_\epsilon^n} \frac{M\Gamma(\alpha+1)}{\Gamma(p)}|S(t_{i+1}) - S(t_i)|^p \\
	    &=& \frac{M\Gamma(\alpha+1)}{\Gamma(p)}\sum_{[t_i,t_{i+1}]\in \pi_n}\1_{\{(t_i,S_{t_i}) \notin \Gamma_f^\epsilon\}}|S(t_{i+1}) - S(t_i)|^p \\
	    &=& \frac{M\Gamma(\alpha+1)}{\Gamma(p)}\mu_n(\{t:(t,S_{t}) \notin \Gamma_f^\epsilon\})
	\end{eqnarray*}
	Since $(\Gamma_f^\epsilon)^c$ is closed, the set $\{t:(t,S_{t}) \notin \Gamma_f^\epsilon\} = \{t:(t,S_{t}) \in (\Gamma_f^\epsilon)^c\}$ is also closed. Then by portmanteau theorem, we have 
	\[
	    \limsup_{n\rightarrow\infty} \mu_n(\{t:(t,S_{t}) \notin \Gamma_f^\epsilon\}) \leq \mu(\{t:(t,S_{t}) \notin \Gamma_f^\epsilon\})
	\]
	Let $\epsilon\rightarrow0$, since $\Gamma_f$ is open, we have
	\[
	    \lim_{\epsilon\rightarrow0} \mu(\{t:(t,S_{t}) \notin \Gamma_f^\epsilon\}) =  \mu(\{t:(t,S_{t}) \notin \Gamma_f\}) = 0
	\]
	By assumption $\ref{AssGammafS}$. Hence 
	\[
	    \lim_{\epsilon\rightarrow0}\lim_{n\rightarrow\infty}\left|\sum_{[t_i,t_{i+1}]\in C_\epsilon^n} \frac{1}{\Gamma(m)}\int_{S(t_i)}^{S( t_{i+1})}(D_x^mf(t_i,r) - D_x^mf(t_i,S(t_i)))(S(t_{i+1}) - r)^{m-1}dr\right| = 0
	\]
	Next we define two sets $C_{\epsilon,1}^n = \{[t_i,t_{i+1}] \notin C_\epsilon^n: S(t_i) \leq S(t_{i+1})\}$ and $C_{\epsilon,2}^n = \{[t_i,t_{i+1}] \notin C_\epsilon^n: S(t_i) > S(t_{i+1})\}$. On the set $C_{\epsilon,1}^n$, we have 
	\begin{eqnarray*}
        &&\left|\sum_{[t_i,t_{i+1}]\in C_{\epsilon,1}^n} \frac{1}{\Gamma(m)}\int_{S(t_i)}^{S( t_{i+1})}(D_x^mf(t_i,r) - D_x^mf(t_i,S(t_i)))(S(t_{i+1}) - r)^{m-1}dr\right| \\
	    &=& \left|\sum_{[t_i,t_{i+1}]\in C^n_{\epsilon,1}}\frac{1}{\Gamma(p)}\int_{S(t_i)}^{S(t_{i+1})}(S(t_{i+1}) - s)^{p-1}C_{S(t_i)^+}^p f(t_i,s)ds\right|
	\end{eqnarray*}
	and on the set $C_{\epsilon,2}^n$, we have 
	\begin{eqnarray*}
        &&\left|\sum_{[t_i,t_{i+1}]\in C_{\epsilon,2}^n} \frac{1}{\Gamma(m)}\int_{S(t_i)}^{S( t_{i+1})}(D_x^mf(t_i,r) - D_x^mf(t_i,S(t_i)))(S(t_{i+1}) - r)^{m-1}dr\right| \\
	    &=& \left|\sum_{[t_i,t_{i+1}]\in C_{\epsilon,2}^n}\frac{1}{\Gamma(m)\Gamma(\alpha)}\int_{S(t_{i+1})}^{S(t_i)}\int_{S(t_{i+1})}^s(S(t_{i+1}) -r)^{m-1}\frac{C_{r^+}^p f(t_i,s)}{(S(t_i) - s)^{1-\alpha}}drds\right|
	\end{eqnarray*}
	We notice that if $(t_i,S_{t_i}) \notin C_\epsilon^n$ and $n$ is large enough, then $(t_i,S_{t_i}),(t_{i+1},S_{t_{i+1}})$ is in the closure of $\Gamma_f^{\epsilon/2}\cup [0,T]\times K$, which is a compact set. Here $K$ is chosen so that ${S_t}_{t\in[0,T]}$ lies in $K$. This means that $C_{x+}^p f(t_i,x)$ is continuous in a neighborhood of $S(t_i)$, hence they are identically zero. Then by uniform continuous property, we can have a modulus of continuity $\omega_\epsilon$ such that 
	\[
	    |C_{a_1^+}^p f(t,b_1) - C_{a_2^+}^p f(t,b_2)| \leq \omega_\epsilon(\sqrt{(a_2-a_1)^2 + (b_2-b_1)^2})
	\]
	for all $(t,a_1),(t,a_2),(t,b_1),(t,b_2) \in \Gamma_f^{\frac{\epsilon}{2}}$. Hence we have for $j = 1,2$
	\begin{eqnarray*}
	    &&\left|\sum_{[t_i,t_{i+1}]\in C_{\epsilon,j}^n} \frac{1}{\Gamma(m)}\int_{S(t_i)}^{S( t_{i+1})}(D_x^mf(t_i,r) - D_x^mf(t_i,S(t_i)))(S(t_{i+1}) - r)^{m-1}dr\right| \\
	    &\leq& \sum_{[t_i,t_{i+1}]\in C_{\epsilon,j}^n}M\omega_\epsilon(S(t_{i+1}) - S(t_i))|S(t_{i+1}) - S(t_i)|^p  \\
	    &\leq& M\omega_\epsilon(\sup\{|S(t_{i+1}) - S(t_i)|\}) \mu_n([0,T])
	\end{eqnarray*}
	for some constant $M$. And let $n\rightarrow\infty$, we see that 
	\[
	    \lim_{n\rightarrow\infty}\left|\sum_{[t_i,t_{i+1}]\in C_{\epsilon,j}^n} \frac{1}{\Gamma(m)}\int_{S(t_i)}^{S( t_{i+1})}(D_x^mf(t_i,r) - D_x^mf(t_i,S(t_i)))(S(t_{i+1}) - r)^{m-1}dr\right| = 0
	\]
	Hence we finally obtain that 
	\[
	    \lim_{n\rightarrow\infty}\left|\sum_{[t_i,t_{i+1}]\in \pi_n} \frac{1}{\Gamma(m)}\int_{S(t_i)}^{S( t_{i+1})}(D_x^mf(t_i,r) - D_x^mf(t_i,S(t_i)))(S(t_{i+1}) - r)^{m-1}dr\right|=0
	\]
	and hence the result.
    \end{proof}
    We can do as the time-independent case to extend the space of functions for fixed path $S$ using approximation in the $C^{1,p}$ space and denote the result space as $E_S$. 
    \begin{remark}
        As   in the time-independent case, we can use a uniform convergence assumption to replace Assumption $\ref{AssGammafS}$:
        \begin{assumption}\label{AssTime}
            There exists a sequence of open set $U_i \in \mathbb{R}^+\times \mathbb{R}$ such that $\bar{U}_{i+1} \subset U_{i}$ and $\lim_{y\rightarrow x} \frac{f^{(m)}(t,y) - f^{(m)}(t,x)}{|y-x|^\alpha} = 0$ locally uniformly outside of $U_i$. Furthermore, Let $C = \cap_{i}^\infty \bar{U}_{i}$, we have 
            \[
                \int_o^T \1_{\{(t,S(t)) \in C\}} d[S](t) = 0
            \]  
        \end{assumption}
        Under this assumption, Theorem $\ref{itoformula3}$ still hold.
    \end{remark}
    \begin{theorem}\label{itoformulatime}
        Let $f \in C^{1,p}([0,T]\times\mathbb{R})$ and $S\in V_p(\pi)$. Under Assumption $\ref{AssTime}$, we have
        \[
            f(t,S(t)) - f(0,S(0)) = \int_0^t \partial_t f(u,S(u))du + \int_0^t D_xf(u,S(u))dS(u)
        \]  
        where $m = \lfloor p\rfloor$ and 
        \[
            \int_0^t D_x f(u,S(u))dS(u) = \lim_{n\rightarrow\infty}\sum_{t_i \in \pi_n}\sum_{j=1}^{m}\frac{f^{(j)}(t\wedge t_i, S(t\wedge t_i))}{j!}(S(t\wedge t_{i+1}) - S(t\wedge t_i))^j.
        \]
    \end{theorem}
    
    \subsection{Extension to path-dependent functionals}\label{sec.pathdependent}
    In this subsection $S_t$ will designate the path stopped at   $t$, i.e. $S_t(s) = S(s\wedge t)$.
    
    We consider the space $D([0,T],\mathbb{R})$ of c\`adl\`ag paths from $[0,T]$ to $\mathbb{R}$. Let
    \[
        \Lambda_T = \{(t,\omega_t):(t,\omega) \in [0,T]\times D([0,T],\mathbb{R})\}
    \]
    be the space of stopped paths. This is a complete metric space equipped with 
    \[
        d_{\infty}((t,\omega),(t',\omega')) := \sup_{s\in[0,T]} |\omega(s\wedge t) - \omega'(s\wedge t')| + |t-t'|
    \]
    We will also need to stop paths 'right before' a given time, and set for $t>0$
    \[
        \omega_{t-}(s) := \left\{
        \begin{array}{cc}
            \omega(s), &  s < t \\
            \lim_{r\uparrow t}, & s\geq t
        \end{array}
        \right.
    \]
    while $\omega_{0-} = \omega_{0}$. We recall here some concepts from the non-anticipative functional calculus \cite{CF10B,cont2012}.
    \begin{definition}
        A non-anticipative functional is a map $F:\Lambda_T\rightarrow \mathbb{R}$. Let $F$ be a non-anticipative functional.
        \begin{itemize}
            \item We write $F \in \mathbb{C}_l^{0,0}$ if for all $t \in [0,T]$ the map $F(t,\cdot): D([0,T],\mathbb{R}) \rightarrow \mathbb{R}$ is continuous and if for all $(t,\omega)\in\Lambda_T$ and all $\epsilon>0$, there exists $\delta >0$ such that for all $(t',\omega') \in \Lambda_T$ with $t' < t$ and $d_{\infty}((t,\omega),(t',\omega')) < \delta$, we have $|F(t,\omega) - F(t',\omega')| < \epsilon$
            \item We write $F \in \mathbb{B}(\Lambda_T)$ if for every $t_0 \in [0,T)$ and every $K > 0$ there exists $C_{K,t_0} > 0$ such that for all $t \in [0,t_0]$ and all $\omega \in D([0,T],\mathbb{R})$ with $\sup_{s\in [0,t]}|\omega(s)| \leq K$, we have $|F(t,\omega)|\leq C_{K,t_0}$
            \item $F$ is horizontally differentiable at $(t,\omega)\in\Lambda_T$ if its horizontal derivative 
            \[
                \mathcal{D}F(t,\omega):=\lim_{h\downarrow 0} \frac{F(t+h,\omega_t) - F(t,\omega_t)}{h}
            \]
            exists. If it exists for all $(t,\omega)\in\Lambda_T$, then $\mathcal{D}F$ is a non-anticipative functional.
            \item $F$ is vertically differentiable at $(t,\omega) \in \Lambda_T$ if its vertical derivative 
            \[
                \nabla_\omega F(t,\omega) := \lim_{h\rightarrow0}\frac{F(t,\omega_t + h\1_{[t,T]}) - F(t,\omega_t)}{h} 
            \]
            exists. If it exists for all $(t,\omega)\in\Lambda_T$, then $\nabla_\omega F$ is a non-anticipative functional. In particular, we define recursively $\nabla_\omega^{k+1} F:=\nabla \nabla_\omega^k F$ whenever this is well defined.
            \item For $p \in \mathbb{N}_0$, we say that $F \in \mathbb{C}_b^{1,p}(\Lambda_T)$ if $F$ is horizontally differentiable and $p$ times vertically differentiable in every $(t,\omega)\in\Lambda_T$, and if $F,\mathcal{D}F,\nabla_\omega^kF \in \mathbb{C}_l^{0,0}(\Lambda_T)\cap\mathbb{B}(\Lambda_T)$ for $k=1,\cdots,p$.
        \end{itemize} 
    \end{definition}
    In order to extend our result to the functional case, we need to give the definition of fractional Caputo derivative of a non-anticipative functional. But here we choose another approach using the uniform convergence of H\"older semi-norm as stated in Remark $\ref{rem.ass4}$. 
    \begin{definition}[Vertical H\"older continuity] Let $0 < \alpha < 1$.\\
    \begin{itemize}
        \item  $F$ is vertically $\alpha-$H\"older continuous at $(t,\omega) \in \Lambda_T$ if 
        \[
            L(t,\omega):=\limsup_{h\rightarrow0}\frac{|F(t,\omega_t + h\1_{[t,T]}) - F(t,\omega_t)|}{|h|^\alpha} < \infty
        \]
        and we call $L(t,\omega)$ the H\"older coefficient of $F$ at $(t,\omega)\in\Lambda_T$
        \item  $F$ is vertically $\alpha-$H\"older continuous on $E \subset \Lambda_T$ if 
        \[
            L:=\lim_{\epsilon\rightarrow0}\sup_{(t,\omega) \in E} \sup_{|h| \leq \epsilon} \frac{|F(t,\omega_t + h\1_{[t,T]}) - F(t,\omega_t)|}{|h|^\alpha} < \infty
        \]
        and we call $L$ the H\"older coefficient of $F$ over $E$ and we denote $F \in \mathbb{C}^{0,\alpha}(E)$.
        \item Let $E_n = \{(t,\omega)\in\Lambda: d((t,\omega),(0,0)) < n\}$.  $F$ is locally vertically $\alpha-$H\"older continuous on $\Lambda_T$ if $F$ is vertically $\alpha-$H\"older continuous on $E_n$ for each $n$. We denote $F \in \mathbb{C}^{0,\alpha}_{loc}$. 
        \item For $m < p < m+1$, we say $F \in \mathbb{C}_{loc}^{0,p}$ if $\nabla_\omega^m F \in \mathbb{C}_{loc}^{0,p-m}$. For $E\subset \Lambda_T$, we say $F \in \mathbb{C}^{0,p}(E)$ if $\nabla_\omega^m F\in\mathbb{C}^{0,p-m}(E)$.
    \end{itemize}
    \end{definition}
    We now give assumptions on the functional $F \in \mathbb{C}_b^{1,m}\cap\mathbb{C}^{0,p}_{loc}$, where $m = \lfloor p\rfloor$. Given a path $S$ which admits $p-$th order variation along a sequence of partitions $\pi$. 
    \begin{assumption}\label{funcass1}
         There exists a sequence of continuous functions $g_k \in \mathbb{C}(\Lambda_T,\mathbb{R})$ with support   $E_k := supp(g_k)$ such that
         \begin{enumerate}\item $E_{k+1}\subset E_k$, $g_k|_{E_{k+1}} = 1, g_{k} \leq 1$ and $F \in \mathbb{C}_{loc}^{0,p}(\Lambda_T- E_k)$ for each $k$ with H\"older coefficient $L = 0$. 
\item 
         There exists a bounded function $g$ on $\Lambda_T$ such that $g_k\rightarrow g$ pointwise and 
         \[
            \int_{0}^T g((t,S_t))d[S]_t = 0
         \]
             \end{enumerate}
    \end{assumption}

    Define the piecewise-constant approximation $S^n$ to $S$ along the partition $\pi_n$:
    \[
        S^n(t) := \sum_{[t_i,t_{i+1}]\in\pi_n}  S(t_{j+1}) \1_{[t_j,t_{j+1})}(t) + S(T)\1_{\{T\}}(t)
    \]
    Then $\lim_{n\rightarrow\infty}||S^n-S||_\infty = 0$ whenever $osc(S,\pi_n)\rightarrow 0$.
    
    \begin{theorem}\label{ItoFunctionalFormula}
        Let $p\notin \mathbb{N}$,  $m = \lfloor p \rfloor$ and $S\in V_p(\pi)$. If $F \in \mathbb{C}_b^{1,m}\cap\mathbb{C}^{0,p}_{loc}$ satisfies Assumption  $\ref{funcass1}$  then the limit
        \[
            \int_0^T \nabla_\omega F(t,S_{t-})d^\pi S(t) := \lim_{n\rightarrow\infty} \sum_{[t_j,t_{j+1}]\pi_n} \sum_{k=1}^m \frac{1}{k!}\nabla_\omega^k F(t_j,S_{t_j-}^n)(S(t_{j+1}\wedge t) - S(t_j\wedge t))^k
        \]
        exists and we have 
        \[
            F(t,S_t) = F(0,S_0) + \int_0^t\mathcal{D}F(s,S_s)ds + \int_0^t \nabla_\omega F(s,S_s)dS(s)
        \]
    \end{theorem}
    \begin{proof}
        We write 
        \begin{eqnarray*}
            F(t,S^n_t) - F(0,S^n_0) &=& \sum_{[t_j,t_{j+1}]\in\pi_n} (F(t_{j+1}\wedge t, S^n_{t_{j+1}\wedge t-})-F(t_j\wedge t,S^n_{t_j\wedge t-})) \\ 
            &&+ F(t,S^n_t) - F(t,S^n_{t-}) \\
            &=& \sum_{[t_j,t_{j+1}]\in\pi_n} (F(t_{j+1}\wedge t, S^n_{t_{j+1}\wedge t-})-F(t_j\wedge t,S^n_{t_j\wedge t-})) + o(1)
        \end{eqnarray*}
        We only need to consider $j$ with $t_{j+1} \leq t$ since the remainder is another $o(1)$ as $n\rightarrow\infty$ by left continuity of functional $F$. Now we split the difference into two parts:
        \[
            F(t_{j+1},S^n_{t_{j+1}-}) - F(t_j,S_{t_j-}^n) = (F(t_{j+1},S^n_{t_{j+1}-}) - F(t_j,S_{t_j}^n))+(F(t_j,S^n_{t_j}) - F(t_j,S_{t_j-}^n))
        \]  
        Now since $S_{t_j}^n = S_{t_{j+1}-}^n$ on $[0,t_{j+1}]$, we have 
        \[
            F(t_{j+1},S^n_{t_{j+1}-}) - F(t_j,S_{t_j}^n) = \int_{t_j}^{t_{j+1}}\mathcal{D}F(u,S_{t_j}^n)du = \int_{t_j}^{t_{j+1}}\mathcal{D}F(u,S^n)du
        \]
        since $\mathcal{D}F$ is non-anticipative functional. Hence we have 
        \[
            \lim_{n\rightarrow\infty} \sum_{[t_j,t_{j+1}]\in\pi_n} (F(t_{j+1},S^n_{t_{j+1}-}) - F(t_j,S_{t_j}^n)) = \int_0^t \mathcal{D}F(u,S_u)du
        \]
        by continuity. It remains to consider the term 
        \[
            F(t_j,S_{t_j}^n) - F(t_j,S_{t_j-}^n) = F(t_j,S_{t_j-}^{n,S_{t_j,t_{j+1}}}) - F(t_j,S_{t_j-}^n)
        \]
        where $S_{t_j,t_{j+1}} = S(t_{j+1}) - S(t_j)$ and $S_{t_j-}^{n,x}(s):=S_{t_j-}^n(s) + \1_{[t_j,T]}(s)x$. Now from Taylor's formula and definition of vertical derivative, we obtain 
        \begin{eqnarray*}
            F(t_j,S_{t_j-}^{n,S_{t_j,t_{j+1}}}) &=& F(t_j,S_{t_j-}^n) + \sum_{k=1}^m \frac{\nabla_\omega^k F(t_j,S_{t_j-}^n)}{k!}(S(t_{j+1}) - S(t_j))^k \\
            &+&\frac{1}{(m-1)!}\int_0^{S_{t_j,t_{j+1}}} (\nabla_\omega^m F(t_j, S_{t_j-}^{n,u}) - \nabla_\omega^m F(t_j,S_{t_j-}^n))(S_{t_j,t_{j+1}} - u)^{m-1}du
        \end{eqnarray*}
        From the definition of $\mathbb{C}^{0,p}$, we see that there exists $M > 0$ such that 
        \[
            |\nabla_\omega^m F(t_j, S_{t_j-}^{n,u}) - \nabla_\omega^m F(t_j,S_{t_j-}^n)| \leq M|u|^{p-m}
        \]
        for each $t_j$ when $n$ is large enough. This shows that
        \[
            \left|\frac{1}{(m-1)!}\int_0^{S_{t_j,t_{j+1}}} (\nabla_\omega^m F(t_j, S_{t_j-}^{n,u}) - \nabla_\omega^m F(t_j,S_{t_j-}^n))(S_{t_j,t_{j+1}} - u)^{m-1}du\right| \leq \tilde{M}|S(t_{j+1}) - S(t_j)|^p
        \]
        for some constant $\tilde{M}$. Fix $k\in\mathbb{N}$, we now divide the partitions into two parts. The first part $\pi_n^1$ contains $[t_j,t_{j+1}]$ such that $(t_j,S_{t_j-}^n) \in E_k$ and the second part $\pi_n^2$ contains $[t_j,t_{j+1}]$ such that $(t_j,S_{t_j-}^n) \notin E_k$. Then we have 
        \begin{eqnarray*}
            &&\sum_{[t_j,t_{j+1}] \in \pi_n} \left|\frac{1}{(m-1)!}\int_0^{S_{t_j,t_{j+1}}} (\nabla_\omega^m F(t_j, S_{t_j-}^{n,u}) - \nabla_\omega^m F(t_j,S_{t_j-}^n))(S_{t_j,t_{j+1}} - u)^{m-1}du\right|  \\
            &\leq& \sum_{[t_j,t_{j+1}] \in \pi_n^1}\tilde{M}|S(t_{j+1}) - S(t_j)|^p + \sum_{[t_j,t_{j+1}] \in \pi_n^2} \tilde{M}L(S_{t_j-}^n, osc(S,\pi_n))|S(t_{j+1}) - S(t_j)|^p 
        \end{eqnarray*}
        where 
        \[
            L(\omega,\epsilon) := \sup_{|h|\leq \epsilon}\frac{|\nabla_\omega^m F(t,\omega_t + h\1_{[t,T]}) - \nabla_\omega^m F(t,\omega_t)|}{|h|^{p-m}}
        \]
        and by the definition of vertically H\"older continuity, we have $\lim_{\epsilon\rightarrow0}\sup_{\omega\in E\slash E_k}L(\omega,\epsilon) = 0$ for each $k$ with $E$ a bounded subset of $\Lambda_T$ containing all stopped paths occurred in this proof. 
        Hence 
        \[
            \lim_{n\rightarrow\infty}\sum_{[t_j,t_{j+1}] \in \pi_n^2} \tilde{M}L(S_{t_j-}^n, osc(S,\pi_n))|S(t_{j+1}) - S(t_j)|^p = 0
        \]
        As for the first part, we have 
        \begin{eqnarray*}
            \lim_{n\rightarrow\infty}\sum_{[t_j,t_{j+1}] \in \pi_n^1}\tilde{M}|S(t_{j+1}) - S(t_j)|^p &\leq& \lim_{n\rightarrow\infty}\sum_{[t_j,t_{j+1}] \in \pi_n}\tilde{M}g_k((t_j,S_{t_j-}^n)|S(t_{j+1}) - S(t_j)|^p\\
            &=& \int_o^t \tilde{M}g_k((t,S_t))d[S]_t
        \end{eqnarray*}
        by assumption \ref{funcass1}. Let $k$ tends to infinity, we get the result.
    \end{proof}
    \begin{remark}
        One may derive theorems \ref{itoformula3} and  \ref{itoformulatime} from above theorem. Let $\Gamma_f$ defined as Assumption \ref{ass5}, $E_k = \{(t,x): d((t,x), \Gamma_f^c) \leq \frac{1}{k}\}$ and define 
        \[
            g((t,\omega)) = \1_{\{(t,\omega(t)) \notin \Gamma_f\}}.
        \]
       Then one may choose smooth functions $y_k$   with support $E_{k-1}$ such that $y_k|_{E_k} = 1, 0\leq y_k \leq 1$ and set
        \[
            g_k(t,\omega) := y_k(t,\omega(t)).
        \]
         $g_k$ and $g$ then satisfy 2) in Assumption \ref{funcass1}.
    \end{remark}
 
    \section{Extension to $\phi$-variation}\label{sec.phi}
    The concept of $\phi-$variation $\phi-$variation along a sequence of partitions  \cite{han2021probabilistic} extends the notion of p-th order variation to more general functions:
    \begin{definition}[$\phi$-variation along a partition sequence]
        A continuous path $S\in\mathbb{C}([0,T],\mathbb{R})$ is said to have a $\phi-$variation along a sequence of partitions $\pi=\left(\pi_n\right)_{n\geq 1}$ if $osc(S,\pi_n)\rightarrow0$ and the sequence of measures
        \[
            \mu^n:=\sum_{t_i\in\pi_n} |\phi(S(t_{i+1}) - S(t_i))|\delta(\cdot-t_i)
        \]
        converges weakly to a measure $\mu$ without atoms. In this case we write $S \in V_{\phi}(\pi)$ and $[S]^\phi(t) = \mu([0,t])$ for $t \in [0,T]$, and we call $[S]^\phi$ the $\phi-$variation of $S$.
    \end{definition}
    The function $\phi$ we considered in this section satisfies some assumption given here.
    \begin{assumption}\label{phiass1}
        There exists $m\in\mathbb{N}$ such that $\phi\in C^m$ and 
        \[
        \lim_{x\rightarrow0} \frac{\phi(x)}{x^m} = 0,\qquad, \lim_{x\rightarrow0} \frac{\phi(x)}{x^{m+1}} = \infty
        \]
        It is obvious that in this case $\phi(0) = 0$. Furthermore, we let $\phi$ map $\mathbb{R}^+$ to $\mathbb{R}^+$ and $\phi$ is either odd or even function.
    \end{assumption}
     We say a function $f$ is $\phi-$continuous if and only if for each compact set $K$, there exists a constant $M$ such that 
    \[
        \frac{|f^{(m)}(y) - f^{(m)}(x)|}{|\phi^{(m)}(y-x)|} \leq M, \qquad x,y\in K
    \]
    We denote the collection of these functions by $C^\phi$.
    \begin{remark}
    As in the fractional case, if we consider the Taylor expansion of $f$ up to '$\phi$', we should write it as 
    \[
        f(y) = f(x) + \sum_{k=1}^m \frac{f^{(k)}}{k!}(y-x)^k + g(x)\phi(y-x) + o(\phi(y-x))
    \]
    it is obvious that if $g$ is continuous, then we would have the remainder term in the change of variable formula. However, for the function $\phi$ satisfying assumption $\ref{phiass1}$, the function $g$ in the Taylor expansion above would be almost surely zero, the same as stated in proposition $\ref{almostsurelyzero}$ for $|x|^p$. What we need is just to change $(y-x)^p$ by $\phi(y-x)$ in the proof of proposition $\ref{almostsurelyzero}$ and notice that if $f$ is differential at some point $x$, then for the case $m=0$,
    \[
        \lim_{y_n\rightarrow x} \frac{f(y_n) - f(x)}{\phi(y_n-x)} = \lim_{y_n\rightarrow x} \frac{f(y_n) - f(x)}{y_n-x}\frac{y_n-x}{\phi(y_n-x)} = 0
    \]
    Follow the steps in the proof of Proposition \ref{almostsurelyzero}, we can show that $g = 0$ almost everywhere. Hence a regular remainder should not appear in this case.
    \end{remark}
    
    We will use the uniform convergence assumption here. 
    \begin{assumption}\label{phiass2}
         There exists a sequence of open sets $U_i$ such that $\bar{U}_{i+1} \subset U_{i}$ for $i=1,2,\cdots$ and $\lim_{y\rightarrow x} \frac{f^{(m)}(y) - f^{(m)}(x)}{|\phi^{(m)}(y-x)|} = 0$ locally uniformly outside of each $U_i$. Furthermore, Let $C = \cap_{i}^\infty \bar{U}_{i}$, we have 
            \[
                \int_0^T \1_{\{S(t) \in C\}} d[S]^\phi(t) = 0
            \]  
    \end{assumption}
    Under the assumption $\ref{phiass2}$, we have the theorem
    \begin{theorem}\label{itophi}
        Let $\phi$ satisfy the assumption $\ref{phiass1}$, $\{\pi_n\}$ be a given sequence of partitions and $S \in V_\phi(\pi)$. Then for every function $f$ satisfies assumption $\ref{phiass2}$ we have the formula 
        \[
            f(S(t)) - f(S(0)) = \int_0^t f'(S(u))dS(u)
        \]
        holds, where the integral 
        \[
            \int_0^t f'(S(u))dS(u) = \lim_{n\rightarrow\infty}\sum_{[t_j,t_{j+1}]\in\pi_n} \sum_{k=1}^m \frac{f^{(k)}(S(t_j))}{k!}(S(t_{j+1}\wedge t) - S(t_j\wedge t))^k
        \]
        is defined as a (pointwise) limit of compensated Riemann sums.
    \end{theorem}
    \begin{proof}
         We have 
        \begin{eqnarray*}
	    f(S_t) - f(S_0) &=& \sum_{[t_i,t_{i+1}]\in\pi_n} f(S_{t\wedge t_{i+1}}) - f(S_{t\wedge t_i}) \\
	    &=&  \sum_{[t_i,t_{i+1}]\in\pi_n}\left(f(S_{t\wedge t_{i+1}}) - f(S_{t\wedge t_i}) - \cdots - \frac{f^{(m)}(S_{t\wedge t_i})}{m!}(S_{t\wedge t_{i+1}} - S_{t\wedge t_i})^n\right) \\
	    &+& \sum_{[t_i,t_{i+1}]\in\pi_n} \sum_{j=1}^m \frac{f^{(j)}(S_{t\wedge t_i})}{j!}(S_{t\wedge t_{j+1}} - S_{t\wedge t_j})^j \\
	    &=& L^n + \sum_{[t_i,t_{i+1}]\in\pi_n} \frac{1}{\Gamma(m)}\int_{S_{t\wedge t_i}}^{S_{t\wedge t_{i+1}}}(f^{(m)}(r) - f^{(m)}(S_{t\wedge t_i}))(S_{t\wedge t_{i+1}} - r)^{m-1}dr
	\end{eqnarray*}
	where $L^n = \sum_{t_i\in\pi_n} \sum_{j=1}^m \frac{f^{(j)}(S_{t\wedge t_i})}{j!}(S_{t\wedge t_{j+1}} - S_{t\wedge t_j})^j$. It is easy to notice that there exists a constant $M>0$ and 
	\begin{eqnarray*}
	    &&\left|\frac{1}{\Gamma(m)}\int_{S_{t\wedge t_i}}^{S_{t\wedge t_{i+1}}}(f^{(m)}(r) - f^{(m)}(S_{t\wedge t_i}))(S_{t\wedge t_{i+1}} - r)^{m-1}dr\right| \\
	    &\leq& \left|M\int_{S_{t\wedge t_i}}^{S_{t\wedge t_{i+1}}}\phi^{(m)}(r-S(t_i\wedge t))(S_{t\wedge t_{i+1}} - r)^{m-1}dr\right|\\
	    &=& \left|M(m-1)\int_{S_{t\wedge t_i}}^{S_{t\wedge t_{i+1}}}\phi^{(m-1)}(r-S(t_i\wedge t))(S_{t\wedge t_{i+1}} - r)^{m-2}dr\right| \\
	    &=& \left|M(m-1)!\int_{S_{t\wedge t_i}}^{S_{t\wedge t_{i+1}}}\phi'(r-S(t_i\wedge t))dr\right| = M(m-1)!|\phi(S(t_{i+1}\wedge t) - S(t_i\wedge t))|
	\end{eqnarray*}
	Now we write 
	\begin{eqnarray*}
	    &&\sum_{[t_i,t_{i+1}]\in\pi_n} \frac{1}{\Gamma(m)}\int_{S_{t\wedge t_i}}^{S_{t\wedge t_{i+1}}}(f^{(m)}(r) - f^{(m)}(S_{t\wedge t_i}))(S_{t\wedge t_{i+1}} - r)^{m-1}dr \\
	    &=& \sum_{S(t_i\wedge t) \in U_k,[t_i,t_{i+1}]\in\pi_n}\cdots + \sum_{S(t_i\wedge t) \notin U_k,[t_i,t_{i+1}]\in\pi_n}\cdots
	\end{eqnarray*}
	We denote first part by $L_1^k$ and second part by $L_2^k$. and by locally uniform convergence property in assumption $\ref{phiass2}$, we have for fixed $\epsilon>0$, all $n$ large enough, 
	\begin{eqnarray*}
	    |L_2^k| &=& \left|\sum_{[t_i,t_{i+1}]\in\pi_n,S(t_i\wedge t)\notin U_k} \frac{1}{\Gamma(m)}\int_{S_{t\wedge t_i}}^{S_{t\wedge t_{i+1}}}(f^{(m)}(r) - f^{(m)}(S_{t\wedge t_i}))(S_{t\wedge t_{i+1}} - r)^{m-1}dr\right| \\
	    &\leq& \epsilon\left|\sum_{[t_i,t_{i+1}]\in\pi_n} \int_{S_{t\wedge t_i}}^{S_{t\wedge t_{i+1}}}\phi^{(m)}(r-S(t\wedge t_i))(S_{t\wedge t_{i+1}} - r)^{m-1}dr\right| \\
	    &\leq& \epsilon (m-1)!\sum_{[t_i,t_{i+1}]\in\pi_n}|\phi(S(t\wedge t_{i+1}) - S(t\wedge t_i))|
	\end{eqnarray*}
	And as for $L_1^k$, we have 
	\begin{eqnarray*}
	    |L_1^k| &\leq& M(m-1)!\sum_{[t_i,t_{i+1}]\in\pi_n, S(t_i\wedge t)\in U_k} |\phi(S(t_{i+1}\wedge t) - S(t_i\wedge t))| \\
	    &\leq& M(m-1)!\sum_{[t_i,t_{i+1}]\in\pi_n}\varphi_k(S(t_i\wedge t)) |\phi(S(t_{i+1}\wedge t) - S(t_i\wedge t))|
	\end{eqnarray*}
	here $\varphi_k$ is a sequence of smooth functions such that \[
	    \varphi_k(x) = \left\{
        \begin{array}{rcl}
        1 & x \in \bar{U}_k\\
        0 & x\notin U_{k-1}
        \end{array} \right.
	\]
	Then $\varphi_k(x) \rightarrow \1_{\{x \in C\}}$ point-wisely as $k\rightarrow \infty$. Combine $L_1^k$ and $L_2^k$, we obtain that 
	\begin{eqnarray*}
	    &&\lim_{n\rightarrow\infty}\left|\sum_{[t_i,t_{i+1}]\in\pi_n} \frac{1}{\Gamma(m)}\int_{S_{t\wedge t_i}}^{S_{t\wedge t_{i+1}}}(f^{(m)}(r) - f^{(m)}(S_{t\wedge t_i}))(S_{t\wedge t_{i+1}} - r)^{m-1}dr\right| \\
	    &\leq& \epsilon (m-1)![S]^\phi(t) + M(m-1)!\int_0^t \varphi_k(S(u))d[S]^\phi(u)
	\end{eqnarray*}
	We then let $\epsilon\rightarrow0$ and $k\rightarrow 0$, obtain that 
	\begin{eqnarray*}
	    &&\lim_{n\rightarrow\infty}\left|\sum_{[t_i,t_{i+1}]\in\pi_n} \frac{1}{\Gamma(m)}\int_{S_{t\wedge t_i}}^{S_{t\wedge t_{i+1}}}(f^{(m)}(r) - f^{(m)}(S_{t\wedge t_i}))(S_{t\wedge t_{i+1}} - r)^{m-1}dr\right| \\
	    &\leq& M(m-1)!\int_0^t \1_{\{S(t)\in C\}}d[S]^\phi(u) = 0
	\end{eqnarray*}
	Hence we get the result.
    \end{proof}
    If we equip space $C^\phi(\mathbb{R})$ with semi-norms 
    \[
        ||f||_{\phi,K} = ||f||_{m,K} + \sup_{x,y \in K}\frac{|f^{(m)}(y) - f^{(m)}(x)|}{|\phi^{(m)}(y-x)|}
    \]
    we can again extend the function space for which theorem holds. But since the assumption of function involving the path $S$, we can again denote this space by $E_S^\phi$.
    \begin{example}
        Here we give an example of application of the above theorem following Han et al \cite{han2021probabilistic}. We take 
        \[
            \phi(x) = \frac{|x|}{\sqrt{-\log(|x|)}}
        \]
        and 
        \[
            S(t) = \sum_{k=0}^\infty \alpha^k\varphi(b^k t), \quad t \in [0,1]
        \]
        where $|\alpha| = \frac{1}{b}$ and $\varphi(t) = \min_{z \in \mathbb{Z}}|z-t|$ or $\varphi(t) = \nu \sin(2\pi t)+ \rho \cos(2\pi t)$. The sequence of partitions considered here is the $b-$adic partitions of $[0,1]$, i.e. 
        \[
            \pi_n = \{kb^{-n}:k=0,\cdots, b^n\}, \quad n \in \mathbb{N}
        \]
        It is shown in \cite{han2021probabilistic} that the $\phi-$variation of $S$ is non-trivial. Furthermore, let $m = 1$, we see that 
        \[
            \lim_{x\rightarrow0}\frac{\phi(x)}{x} = 0, \quad \lim_{x\rightarrow0}\frac{\phi(x)}{x^2} =\infty
        \]
        Hence theorem $\ref{itophi}$ applies to the path $S$ and partitions $\pi$ with $\phi = \frac{|x|}{\sqrt{-\log (|x|)}}$. And here we can choose  $f(x) = \phi(x)$ which satisfies the assumption \ref{phiass2} with $C = \{0\}$.
    \end{example}
 
 \section{An isometry formula for the pathwise integral}\label{sec.isometry}
Ananova and Cont	\cite{ananova2017} proved a pathwise analogue of the It\^o isometry for the \follmer\  integral with respect to paths with finite quadratic variation. This relation was extended to  $p\geq 1$ in \cite{cont2018pathwise}. 
In the same flavor we derive here an  isometry relation for the pathwise integral in terms of the $\phi-$variation, where $\phi$ is a more general function satisfying the following assumption: 
\begin{assumption}\label{phiass3}
	     For $x \geq 0$, we have 
	     \begin{enumerate}
	               \item $\phi$ is strictly increasing, continuous and $\phi(0) = 0$
	         \item $\phi$ is a convex function on $[0,\infty)$ and $\log\phi(x)$ is a convex function of $\log x$.
	         \item For each $x > 0$, the limit
	         \[
	            \varphi(x):=\lim_{y\rightarrow 0}\frac{\phi(xy)}{\phi(y)}
	         \]
	         exists and the convergene is uniform on bounded sets.
	         \item $\infty> p(\phi) = \sup\left\{p: \lim_{x\rightarrow0}\frac{\phi(x)}{|x|^p} = 0\right\}>0$.
	     \end{enumerate}
	\end{assumption}
	We first introduce a generalized version of Minkowski's inequality here\cite{mulholland1949generalizations}.
	\begin{lemma}\label{generalizedMinkowski}
	    Suppose $\phi$ satisfies top 2 conditions in assumption $\ref{phiass3}$, then we have inequality
	    \[
	        \phi^{-1}\left(\sum \phi(a_n+b_n)\right) \leq \phi^{-1}\left(\sum\phi(a_n)\right) + \phi^{-1}\left(\sum\phi(b_n)\right)
	    \]
	    for two positive sequences $\{a_n\}$ and $\{b_n\}$.
	\end{lemma}
	Using this lemma, we can show that
	\begin{theorem}
	    Let $\phi$ be an even function satisfying Assumption $\ref{phiass3}$ and $\pi_n$  a sequence of partitions of $[0,T]$ with vanishing mesh. Suppose $S \in V_\phi(\pi)\cap C^\alpha([0,T],\mathbb{R})$ for $\alpha > \frac{\sqrt{1+\frac{4}{p(\phi)}} - 1}{2}$. Let $F \in \mathbb{C}_b^{1,2}(\Lambda_T)$ such that $\nabla_\omega F\in\mathbb{C}_b^{1,1}(\Lambda_T)$. Assume furthermore that $F$ is Lipschitz-continuous with respect to $d_\infty$. Then $F(\cdot,S) \in V_\phi(\pi)$ and 
	    \[
	        [F(\cdot,S)]^\phi(t) = \int_{0}^t\varphi(|\nabla_\omega F(s,S_s)|)d[S]^\phi(s)
	    \]
	\end{theorem}
	\begin{proof}
	    We have the same
	    \[
	        R_F(s,t) := |F(t,S_t) - F(s,S_t) - \nabla_\omega F(s.S_s)(S(t)-S(s))| \leq C|t-s|^{\alpha + \alpha^2}
	    \]
	    with $\gamma_F(s,t) := \nabla_\omega F(s.S_s)(S(t)-S(s))$, we have from lemma \ref{generalizedMinkowski} that 
	    \begin{eqnarray*}
	        &&\phi^{-1}\left(\sum_{\substack{[t_j,t_{j+1}]\in\pi_n\\    t_{j+1} \leq t}} \phi(|F(t_{j+1},S_{t_{j+1}}) - F(t_j,S_{t_j})|)\right)\\
	        &=& \phi^{-1}\left(\sum_{\substack{[t_j,t_{j+1}]\in\pi_n\\    t_{j+1} \leq t}} \phi(|R_F(t_j,t_{j+1}) + \gamma_F(t_j,t_{j+1})|)\right) \\
	        &\leq& \phi^{-1}\left(\sum_{\substack{[t_j,t_{j+1}]\in\pi_n\\    t_{j+1} \leq t}}\phi(|R_F(t_j,t_{j+1})|)\right) + \phi^{-1}\left(\sum_{\substack{[t_j,t_{j+1}]\in\pi_n\\    t_{j+1} \leq t}}\phi(|\gamma_F(t_j,t_{j+1})|)\right) \\
	        &\leq& 2 \phi^{-1}\left(\sum_{\substack{[t_j,t_{j+1}]\in\pi_n\\    t_{j+1} \leq t}}\phi(|R_F(t_j,t_{j+1})|)\right) + \phi^{-1}\left(\sum_{\substack{[t_j,t_{j+1}]\in\pi_n\\    t_{j+1} \leq t}}\phi(|R_F(t_j,t_{j+1}) + \gamma_F(t_j,t_{j+1})|)\right)
	    \end{eqnarray*}
	    and since $R_F(s,t) \leq C|t-s|^{\alpha^2 + \alpha}$, we have
	    \[
	         \phi^{-1}\left(\sum_{\substack{[t_j,t_{j+1}]\in\pi_n\\    t_{j+1} \leq t}}\phi(|R_F(t_j,t_{j+1})|)\right) \leq \phi^{-1}\left(\sum_{\substack{[t_j,t_{j+1}]\in\pi_n\\    t_{j+1} \leq t}}\phi(C|t_{j+1} - t_j|^{\alpha + \alpha^2})\right)
	    \]
	    Furthermore due to $\alpha > \frac{\sqrt{1+\frac{4}{p(\phi)}} - 1}{2}$, we know 
	    \[
	        \alpha + \alpha^2 > \frac{1}{p(\phi)}
	    \]
	    Hence there exist $\epsilon > 0$ such that
	    \[
	        \phi(C|t_{j+1} - t_j|^{\alpha + \alpha^2}) \leq \phi(C|t_{j+1} - t_j|^{\frac{1}{p(\phi)-\epsilon}})
	    \]
	    Then by the definition of $p(\phi)$, we see that 
	    \[
	         \phi(C|t_{j+1} - t_j|^{\frac{1}{p(\phi)-\epsilon}}) \leq \omega(|t_{j+1} - t_j|)C^{p(\phi) - \epsilon/2}|t_{j+1} - t_j|^{\frac{p(\phi) - \epsilon/2}{p(\phi) - \epsilon}}
	    \]
	    where $\omega$ is continuous on $\mathbb{R}^+$ and $\omega(0) = 0$. Combine all above, we get 
	    \[
	         \lim_{n\rightarrow\infty}  \phi^{-1}\left(\sum_{\substack{[t_j,t_{j+1}]\in\pi_n\\    t_{j+1} \leq t}}\phi(|R_F(t_j,t_{j+1})|)\right) =0
	    \]
	    and we also have
	    \[
	        \lim_{n\rightarrow\infty} \phi^{-1}\left(\sum_{\substack{[t_j,t_{j+1}]\in\pi_n\\    t_{j+1} \leq t}} \phi(|\gamma_F(t_j,t_{j+1})|)\right) =\phi^{-1}\left(\int_0^t\varphi(|\nabla_\omega F(s,S_s)|)d[S]^\phi(s)\right)
	    \]
	    In fact, since $\nabla_\omega F(t,S_t) \in \mathbb{B}(\Lambda_T)$, there exist $M>0$ such that $|\nabla_\omega F(u,S_u)| \leq M$. For each $\epsilon >0$, there exists $\delta > 0$ such that 
	    \[
	        \left|\frac{\phi(xy)}{\phi(y)} - \varphi(x)\right| \leq \epsilon
	    \]
	    for all $|x| \leq M, |y|\leq \delta$. Then we have for $n$ large enough(i.e. $osc(S,\pi_n) \leq \delta$)
	    \begin{eqnarray*}
	        &&\sum_{[t_j,t_{j+1}] \in \pi_n, t_{j+1}\leq t} \phi(|\nabla_\omega F(t_j,S_{t_j})(S(t_{j+1}) - S(t_j))|) \\
	        &\leq& \sum_{[t_j,t_{j+1}] \in \pi_n, t_{j+1}\leq t} \varphi(|\nabla_\omega F(t_j,S_{t_j})|)\phi(|S(t_{j+1}) - S(t_j)|) + \epsilon\phi(|S(t_{j+1}) - S(t_j)|)
	    \end{eqnarray*}
	    Then let $n$ tends to infinity and $\epsilon$ tends to $0$, we obtain
	    \[
	        \lim_{n\rightarrow\infty}\sum_{[t_j,t_{j+1}] \in \pi_n, t_{j+1}\leq t} \phi(|\nabla_\omega F(t_j,S_{t_j})(S(t_{j+1}) - S(t_j))|) = \int_0^t\varphi(|\nabla_\omega F(s,S_s)|)d[S]^\phi(s)
	    \]
	    Hence we get 
	    \[
	        [F(\cdot,S)]^\phi(t) = \int_0^t\varphi(|\nabla_\omega F(s,S_s)|) d[S]^\phi(s)
	    \]
	\end{proof}
	In particular, if we set $\phi(x) = |x|^p$ for $p > 0$, we have 
	\[
	    \varphi(x) = \lim_{y\rightarrow0}\frac{|xy|^p}{|y|^p} = |x|^p
	\]
	Hence, we have the corollary 
	\begin{corollary}
	    Let $p \in \mathbb{R}^+$, and $\alpha > ((1+\frac{4}{p})^\frac{1}{2}-1)/2$, let $\pi_n$ be a sequence of partitions with vanishing mesh, and let $S \in V_p(\pi)\cap C^\alpha([0,T],\mathbb{R})$. Let $F \in \mathbb{C}_b^{1,2}(\Lambda_T)$ such that $\nabla_{\omega} F \in \mathbb{C}_b^{1,1}(\Lambda_T))$. Assume furthermore that $F$ is Lipschitz-continuous with respect to $d_\infty$. Then $F(\cdot,S) \in V_p(\pi)$ and 
	    \[
	        [F(\cdot,S)]^p(t) = \int_0^t|\nabla_\omega F(s,S_s)|^p d[S]^p(s)
	    \]
	\end{corollary}
	\begin{remark}
	    There exists a non empty set of $p$ and $\alpha$ such that assumption $S\in V_{\phi}\cap^\alpha([0,T],\mathbb{R})$ for $\alpha > \frac{\sqrt{1 + \frac{4}{p(\phi)}}-1}{2}$ is valid. In fact, what we need to verify is that 
	    \[
	        \sum_{[t_i,t_{i+1}]\in\pi_n} \phi(S(t_{i+1}) - S(t_i))
	    \]
	    should not be identical zero for all $\alpha > \frac{\sqrt{1 + \frac{4}{p(\phi)}}-1}{2}$. Due to the definition of $p(\phi)$, $\phi(S(t_{i+1}) - S(t_i))$ should behave like $C(\pi_n)|S(t_{i+1}) - S(t_i)|^{p(\phi)}$ when the oscillation of $\pi_n$ is small enough. Here $C(\pi_n)$ are uniformly bounded. Hence it remains to check 
	    \[
	        \sum_{[t_i,t_{i+1}]\in\pi_n} |S(t_{i+1}) - S(t_i)|^{p(\phi)}
	    \]
	    can not be zero for all $S \in C^\alpha$. By the definition of H\"older property, it remains to verify 
	    \[
	        p(\phi) \frac{\sqrt{1 + \frac{4}{p(\phi)}}-1}{2} < 1
	    \]
	    which is equivalent to 
	    \[
	        1 + \frac{4}{p(\phi)} < \left(1 + \frac{2}{p(\phi)}\right)^2
	    \]
	    which is obviously true.
	\end{remark}
	\begin{example}
	    We can see $\phi(x) = \frac{x}{\sqrt{-\log x}}$   satisfies Assumption \eqref{phiass3}. In fact, it is strictly increasing and continuous and $\phi(0) = 0$. Furthermore, 
	    \[
	        \phi'(x) = \frac{1}{\sqrt{-\log x}} + \frac{1}{2(\sqrt{-\log x})^3}>0,\quad {\rm and}\quad
	        \phi''(x) = \frac{1}{2x(\sqrt{-\log x})^3} + \frac{3}{4x(\sqrt{-\log x})^5} > 0
	    \]
	    Hence $\phi$ is convex. And we have 
	    \[
	        \log\phi(x) = \log x-\frac{\log (-\log x)}{2}
	    \]
	    It is easy to check that $x \rightarrow x - \frac{\log(-x)}{2}$ is convex. And 
	    \begin{eqnarray*}
	        \varphi(x) &=& \lim_{y\rightarrow0}\frac{\phi(|x||y|)}{\phi(|y|)}=
	          \lim_{y\rightarrow0} |x| \sqrt{\frac{-\log y}{-\log x - \log y}}= |x|
	    \end{eqnarray*}
	    Hence we have 
	    \[
	        [F(\cdot,S)]^\phi(t) = \int_0^t |\nabla_\omega F(s,S_s)|d[S]^\phi(s)
	    \]
	\end{example}   
    \section{Multi-dimensional extensions}\label{sec.multidimensional}
    
    The extension of the above results to the multidimensional case is not entirely straightforward, as the space $V_p(\pi)$ is not a vector space \cite{schied2016}.
    
    In the case of integer $p$, some definitions may be extended to the vector case by considering symmetric tensor-valued measures
    as in
    \cite{cont2018pathwise} but this is not convenient for the fractional case.
    We  use a definition which is equivalent to the definition in the paper \cite{cont2018pathwise} when $p$ is an integer. 
    
    \begin{definition}\label{multidef}
        Let $p\geq 1$ and $S = (S^1,\cdots,S^d) \in C([0,T],\mathbb{R}^d)$ be a continuous path. Let $\{\pi_n\}$ be a sequence of partition of $[0,T]$. We say that $S$ has a $p-$the order variation along $\pi = \{\pi_n\}$ if $osc(S,\pi_n) \rightarrow 0$ and for any linear combination $S_a := \sum_{i=1}^d a_i S^i$, we have $S_a \in V_p(\pi)$. Furthermore, we denote $\mu_{S_a}$ to be the weak limit of measures
        \[
            \mu_{S_a}^{n}:=\sum_{[t_j,t_{j+1]}\in\pi_n} \delta(\cdot - t_j)|S_a(t_{j+1}) - S_a(t_j)|^p
        \]
    \end{definition}
    \begin{remark}
        Although we can define the situation when $S$ admits a $p$-th variation, we can not give an object directly related to the $p$-th variation of the path $S$ in this definition. 
        
        However, it can be seen that multi-dimension fractional Brownian motion will definitely admit $p-$th variation along  sequences of partitions with fine enough mes.
    \end{remark}

    We now state the assumptions on the function $f$ that will appear in the It\^o formula.
    \begin{assumption}\label{multiass1}
         Let $p = m + \alpha$ with $m = \lfloor p \rfloor$. The symmetric tensor product $\nabla^m f(x) \in C_{loc}^\alpha(Sym_m(\mathbb{R}^d))$. Furthermore, there exists a sequence of open sets $U_k$ such that $\bar{U}_{k+1} \subset U_k$ and 
         \[
            \lim_{y\rightarrow x}\frac{||\nabla^m f(y) - \nabla^m f(x)||}{||y-x||^\alpha} = 0
         \]
         locally uniformly on $U_k^c$.
    \end{assumption}
\begin{assumption}\label{multiass2}
     The sequence $U_k$ in assumption \ref{multiass1} converges to a set $C = \cap_{k} U_k$ and for all $a = (a_1, \cdots, a_d) \in \mathbb{R}^d$,
     \[
        \int_{0}^T \1_{\{S(t) \in C\}} d\mu_{S_a} = 0
     \]
     where $\mu_{S_a}$ is defined as in Definition $\ref{multidef}$.
\end{assumption}
\begin{theorem}\label{itoformulamulti}
    Suppose $S \in V_p(\pi)$ and $f$ satisfies assumptions $\ref{multiass1}$ and $\ref{multiass2}$. Then we have the formula 
    \[
        f(S(t)) - f(S(0)) = \int_0^t \nabla f(S(u))dS(u)
    \]
    where 
    \[
        \int_0^t \nabla f(S(u))dS(u):=\lim_{n\rightarrow \infty} \sum_{[t_j,t_{j+1}]\in\pi_n} \sum_{k=1}^m\frac{1}{k!} <\nabla^k f(S(t_j)), (S(t_{j+1}) - S(t_j)^{\otimes k}>
    \]
\end{theorem}
Before we prove the theorem, we give a lemma here.
\begin{lemma}\label{MultiLemma}
    Let $\alpha_1,\cdots,\alpha_d$ be positive numbers such that $\alpha = \sum_{i=1}^d\alpha_i > 1$. Suppose $S \in V^\alpha(\pi)$ in the sense of definition $\ref{multidef}$. Then the limit of  
    \[
        \sum_{[t_j,t_{j+1}]\in\pi_n} \Pi_{i=1}^d|S^i(t_{j+1}) - S^i(t_j)|^{\alpha_i}
    \]
    can be bounded by sum of $\alpha-$th variations of $2^d$ different linear combinations of components of path $S$.
\end{lemma}
\begin{proof}
    For any positive number $\alpha_1,\cdots\alpha_d$ such that $\alpha = \sum_{i=1}^d\alpha_i > 1$, Young's inequality tells us that 
    \[
        \Pi_{i=1}^d |S^i(t_{j+1}) - S^i(t_j)|^{\frac{\alpha_i}{\alpha}} \leq \sum_{i=1}^d \frac{\alpha_i}{\alpha}|S^i(t_{j+1}) - S^i(t_j)|
    \]
    which shows
    \[
        \Pi_{i=1}^d |S^i(t_{j+1}) - S^i(t_j)|^{\alpha_i} \leq \sum_{\epsilon_i = \pm 1}|\sum_{i=1}^d \epsilon_i\frac{\alpha_i}{\alpha} (S^i(t_{j+1}) - S^i(t_j))|^\alpha
    \]
    By taking sum over the partition $\pi_n$ and taking the limit of $n$ we then obtain desired result.
\end{proof}

\begin{proof}[Proof of theorem \ref{itoformulamulti}]
    The technique for the proof is the same as the previous theorems. We only consider the situation $t=T$, the case $t<T$ is similar with an $o(1)$ additional term. Using a  Taylor expansion,
    \begin{eqnarray*}
        f(S(T)) - f(S(0)) &=& \sum_{[t_j,t_{j+1}]\in\pi_n} (f(S(t_{j+1})) - f(S(t_j)))\\
        &=& \sum_{[t_j,t_{j+1}] \in \pi_n} \sum_{k=1}^m\frac{1}{k!}<\nabla^m f(S(t_j)),(S(t_{j+1}) - S(t_j))^{\otimes k}> \\
        &+& \frac{1}{(m-1)!}\int_0^1 <\nabla^m f(S(t_j) + \lambda(S(t_{j+1}) - S(t_j))) - \nabla^m f(S(t_j))\\
        &&,(S(t_{j+1}) - S(t_j))^{\otimes m}>(1-\lambda)^{m-1}d\lambda
    \end{eqnarray*}
    We denote last integral by $R(t_j,t_{j+1})$. By H\"older continuity of $\nabla^m f$, there exists $M > 0$ such that 
    \[
        \frac{||\nabla^m f(S(t_j) + \lambda(S(t_{j+1}) - S(t_j))) - \nabla^m f(S(t_j))||}{\lambda^\alpha||S(t_{j+1}) - S(t_j)||^\alpha} \leq M
    \] 
    Then we would get 
    \begin{eqnarray*}
        |R(t_j,t_{j+1})| \leq \frac{M}{(m-1)!}\int_0^1 \lambda^\alpha ||S(t_{j+1}) - S(t_j)||^\alpha||(S(t_{j+1}) - S(t_j))^{\otimes m}||(1-\lambda)^{m-1}d\lambda
    \end{eqnarray*}
     Fixed $k\in\mathbb{N}$, We divide the partition $\pi_n$ into two parts $\pi^1_n :=\{[t_j,t_{j+1}] \in \pi_n: S(t_j) \in U_k\}$ and $\pi_n^2: = \pi_n\slash \pi_n^1$. Then we have on $\pi_n^1$,
    \[
        \sum_{[t_j,t_{j+1}]\in\pi_n^1} |R(t_j,t_{j+1})| \leq \sum_{[t_j,t_{j+1}]\in\pi_n^1}  \tilde{M}||S(t_{j+1}) - S(t_j)||^\alpha||(S(t_{j+1}) - S(t_j))^{\otimes m}||
    \]
    with $\tilde{M} = \frac{M}{(m-1)!}\int_0^1 \lambda^\alpha(1-\lambda)^{m-1}d\lambda$ and on $\pi_n^2$, for each $\epsilon>0$, when $n$ is large enough,
    \[
        \sum_{[t_j,t_{j+1}]\in\pi_n^2} |R(t_j,t_{j+1})| \leq \sum_{[t_j,t_{j+1}]\in\pi_n^2} \epsilon||S(t_{j+1}) - S(t_j)||^\alpha||(S(t_{j+1}) - S(t_j))^{\otimes m}||
    \]
    In order to give our result, we need to give a bound for
    \[
        \lim_{n\rightarrow \infty} \sum_{[t_j,t_{j+1}]\in\pi_n} ||S(t_{j+1}) - S(t_j)||^\alpha||(S(t_{j+1}) - S(t_j))^{\otimes m}||
    \]
    First we have 
    \[
        ||S(t_{j+1}) - S(t_j)||^\alpha \leq C\sum_{i=1}^d |S^i(t_{j+1}) - S^i(t_j)|^\alpha
    \]
    for some constant $C$. By Lemma $\ref{MultiLemma}$, we know that each component of 
    \[
         \lim_{n\rightarrow \infty} \sum_{[t_j,t_{j+1}]\in\pi_n} ||S(t_{j+1}) - S(t_j)||^\alpha(S(t_{j+1}) - S(t_j))^{\otimes m}
    \]
    can be bounded by the sum of $2^d$ different $p-$th varaitions of linear combinations of path $S$. Hence 
    \[
        \lim_{n\rightarrow \infty} \sum_{[t_j,t_{j+1}]\in\pi_n} ||S(t_{j+1}) - S(t_j)||^\alpha||(S(t_{j+1}) - S(t_j))^{\otimes m}||
    \]
    is bounded by some constant $C(S,\pi)$ related to path $S$. Furthermore, in the sense of measure, it will be dominant by sum of measures $\mu_r$ with enough numbers $r = 1,\cdots, R$.  Then we have 
    \begin{eqnarray*}
        &&\lim_{n\rightarrow\infty} \sum_{[t_j,t_{j+1}]\in\pi_n} |R(t_j,t_{j+1})| \\
        &\leq& \lim_{n\rightarrow\infty} \sum_{[t_j,t_{j+1}]\in\pi_n} \tilde{M}g_{k}(S(t_j))||S(t_{j+1}) - S(t_j)||^\alpha||(S(t_{j+1}) - S(t_j))^{\otimes m}|| + \epsilon C(S,\pi)\\
        &\leq& \sum_{r=1}^R \int_0^T g_{k}(S(t)) d\mu_r + \epsilon C(S,\pi)
    \end{eqnarray*}
    Here $g_k$ is a positive smooth function with support on $U_{k-1}$, $g_k|_{U_k} = 1$ and $g_k\leq 1$. Letting $k\rightarrow\infty$ and $\epsilon\rightarrow0$, by assumption $\ref{multiass2}$, we obtain the result.
\end{proof}
\begin{remark}
    This result may be extended to the time-dependent functions as discussed above.
\end{remark}
\begin{remark}
    It is hopeless to obtain a multi-variate Taylor expansion up to order $p$ in general. The  simplest function $f(x) = |x|^\alpha$ with $0< \alpha < 1$ does not admit a Taylor-type expansion up to order $\alpha$ at $0$ since the limit
    \[
        \lim_{x \rightarrow 0}\frac{||x||^\alpha}{|x_1|^\alpha + \cdots |x_d|^\alpha}
    \]
    does not exist.
\end{remark}


\end{document}